\documentclass[12pt, reqno]{amsart}
\usepackage{amsmath}
\usepackage{amssymb}
\usepackage{epsfig}
\usepackage{graphicx}
\usepackage{color}
\usepackage{fullpage}
\definecolor{shadecolor}{gray}{0.875}
\usepackage{amscd}

\numberwithin{equation}{section}

\input xy
\xyoption{all}

\calclayout
\allowdisplaybreaks[3]

\theoremstyle{plain}
\newtheorem{prop}{Proposition}[section]

\newtheorem{theo}[prop]{Theorem}
\newtheorem{coro}[prop]{Corollary}

\newtheorem{lemm}[prop]{Lemma}

\theoremstyle{definition}
\newtheorem{defi}[prop]{Definition}

\newtheorem{ques}[prop]{Question}
\newtheorem{conj}[prop]{Conjecture}

\newtheorem*{mconj}{Manin's Conjecture}

\newtheorem{rema}[prop]{Remark}

\newtheorem{exam}[prop]{Example}

\def\ra{\rightarrow}

\def\bR{{\mathbb R}}

\def\vol{\mathrm{Vol}}
\def\Chow{\mathrm{Chow}}

\def\Hilb{\mathrm{Hilb}}
\def\Supp{\mathrm{Supp}}
\def\Sym{\mathrm{Sym}}

\makeatother
\makeatletter

\author{Brian Lehmann}
\address{Department of Mathematics \\
Boston College  \\
Chestnut Hill, MA \, \, 02467}
\email{lehmannb@bc.edu}

\author{Sho Tanimoto }
\address{Department of Mathematical Sciences\\
University of Copenhagen\\
Universitetspark 5\\
2100 Copenhagen $\emptyset$\\
Denmark}
\email{sho@math.ku.dk}

\title[The geometry of thin exceptional sets]{On the geometry of thin exceptional sets in Manin's conjecture}

\begin{document}
\date{\today}

\begin{abstract}
Manin's Conjecture predicts the rate of growth of rational points of a bounded height after removing those lying on an exceptional set.  We study whether the exceptional set in Manin's Conjecture is a thin set using the minimal model program and boundedness of log Fano varieties. 
\end{abstract}

\maketitle

\section{Introduction} 
\label{secct:intro}

Let $X$ be a smooth projective variety defined over a number field $F$ and let $\mathcal L=(L, \|\cdot\|)$ be a big and nef adelically metrized line bundle on $X$ with associated height function $H_{\mathcal{L}}$.  For any subset $Q \subset X(F)$, define the counting function
$$
N(Q,\mathcal L, B):=\#\{ x\in Q \, |\, H_{\mathcal L}(x)\le B\}.
$$
Manin's Conjecture predicts that the asymptotic behavior of the counting function as $B$ increases is controlled by certain geometric invariants of $(X,\mathcal{L})$.  Let $\Lambda_{\mathrm{eff}}(X) \subset \mathrm{NS}(X)_{\mathbb R}$ denote the cone of pseudo-effective divisors.
Define
\begin{equation*}
a(X,L) = \min \{ t\in \bR \mid t[L] + [K_X] \in \Lambda_{\mathrm{eff}}(X) \}.
\end{equation*}
and
\begin{align*}
b(F, X,L) = & \textrm{ the codimension of the minimal} \\
& \textrm{supported face of }  \Lambda_{\mathrm{eff}}(X) \textrm{ containing} \\
& \textrm{the numerical class } a(X, L)[L] + [K_X].
\end{align*}
Recall that a thin subset of $X(F)$ is a finite union $\cup_{i} \pi_{i}(Y_{i}(F))$, where each $\pi_{i}: Y_{i} \to X$ is a morphism that is generically finite onto its image and admits no rational section.  The following version of Manin's Conjecture was first suggested by Peyre in \cite[Section 8]{Peyre03} and explicitly stated in \cite[Conjecture 1.4]{BL16} and \cite[Conjecture 1.1]{LeRudulier}: 
\begin{mconj}
Let $X$ be a smooth projective variety over a number field $F$ with ample anticanonical class $-K_X$.  
Let $\mathcal{L} = (L,\Vert \cdot \Vert)$ be a big and nef adelically metrized line bundle on $X$. 
There exists a thin set $Z \subset X(F)$
such that one has
\begin{equation}
\label{eqn:manin}
N(X(F) \smallsetminus Z, 
\mathcal L, B)
\sim c(F, X(F) \smallsetminus Z, 
\mathcal{L})B^{a(X,L)} \log(B)^{b(F, X,L)-1}, \quad B\ra \infty,
\end{equation}
where $c(F, X(F) \smallsetminus Z, 
\mathcal{L})$ is Peyre's constant as in \cite{Peyre} and \cite{BT}.
\end{mconj}

See \cite{Peyre16} for a different version of Manin's conjecture using the notion of freeness.
Previous statements of the conjecture (\cite{BM}, \cite{Peyre}, \cite{BT}) gave a ``closed set'' version: instead of removing a thin set $\cup_{i} \pi_{i}(Y_{i}(F))$, they removed only the points lying on a closed subset $V(F)$.  It turns out that the closed set version of Manin's Conjecture is false.  All known counter-examples to the closed set version of Manin's Conjecture for $(X,L)$ arise from two kinds of geometric incompatibilities:
\begin{itemize}
\item a Zariski dense set of subvarieties $Y$ such that $$(a(Y,L|_{Y}),b(F, Y,L|_{Y})) > (a(X,L),b(F, X,L))$$ in the lexicographic order (as in \cite{BT-cubic}), or
\item a generically finite dominant morphism $f: Y \to X$ such that $$(a(Y,f^{*}L),b(F, Y,f^{*}L)) > (a(X,L),b(F, X,L))$$ in the lexicographic order (as in \cite{LeRudulier}).
\end{itemize}
In either case, the expected growth rate of points of bounded height on $Y$ is larger than that on $X$ and the pushforward of these points can form a dense set on $X$.

\bigskip

The following conjecture predicts that such geometric incompatibilities can not obstruct the thin set version of Manin's Conjecture.

\begin{conj} \label{mainconj}
Let $X$ be a smooth uniruled variety over a number field $F$ and let $L$ be a big and nef $\mathbb{Q}$-divisor on $X$.  Consider $F$-morphisms $f: Y \to X$ where $Y$ is a smooth projective variety and $f$ is generically finite onto its image.  As we vary over all such morphisms $f$ such that either $f^{*}L$ is not big or
\begin{equation*}
(a(Y,f^{*}L),b(F, Y,f^{*}L)) > (a(X,L),b(F, X,L))
\end{equation*}
in the lexicographic order, the points
\begin{equation*}
\bigcup_{f} f(Y(F))
\end{equation*}
are contained in a thin subset of $X(F)$.
\end{conj}

\begin{rema}
We also address the case when
\begin{equation*}
(a(Y,f^{*}L),b(F, Y,f^{*}L)) = (a(X,L),b(F, X,L)).
\end{equation*}
However, the image of points on such varieties $Y$ need not lie in a thin set.  In fact, to obtain the correct Peyre's constant one may need to allow contributions of such points.  We give an example of this behavior in Example~\ref{exam: Sano}.
\end{rema}

We make partial progress towards Conjecture \ref{mainconj}, building on \cite{BT}, \cite{HTT15}, \cite{LTT14}, \cite{HJ16}.  Our approach can be broken down into two steps.
\begin{enumerate}
\item First, we prove a boundedness statement over $\overline{F}$ using the minimal model program.
\item Second, we prove a thinness statement over $F$ using Hilbert's Irreducibility Theorem.
\end{enumerate}
We emphasize that our results are not just theoretical, but can be used to understand the exceptional set in specific examples.  We discuss the example of \cite{BT-cubic} (Example \ref{btexample}) and the example of \cite{LeRudulier} (Examples \ref{lerudulierexample2} and \ref{lerudulierexample}), as well as many others in Sections \ref{surfacesec}, \ref{fanothreefoldsec}, and \ref{examplesec}.

\begin{exam}
Let $X$ be a smooth Fano threefold over a number field $F$ which has geometric Picard rank $1$ and index $1$, degree $\geq 10$, and is general in moduli.  Then:
\begin{itemize}
\item The set of $-K_{X}$-lines (i.e.~curves with $-K_{X}.C=1$) sweeps out a divisor $D$ on $X$.
\item The set of $-K_{X}$-conics (i.e.~curves with $-K_{X}.C=2$) is parametrized by a surface $S$, forming a family $\pi: \mathcal{C} \to S$, and the morphism $f: \mathcal{C} \to X$ is generically finite dominant of degree $\geq 2$.
\end{itemize}
Corollary \ref{coro: fano3fold} shows that the conclusion of Conjecture \ref{mainconj} holds on $X$ with thin set $D(F) \cup f(\mathcal{C}(F)) \cup f(E)(F)$ where $E$ is the ramification divisor of $f$.
\end{exam}

For the rest of the introduction, we describe more precisely the geometry underlying thinness of point contributions in certain situations.

\subsection{Families of subvarieties}

The behavior of the $a$-constant for subvarieties was worked out by \cite{LTT14} and \cite{HJ16}.  Using the recent results of \cite{DiCerbo16}, \cite{HJ16} proves that over an algebraically closed field of characteristic $0$ there is a closed proper subset $V \subset X$ such that all subvarieties $Y$ with $a(Y,L) > a(X,L)$ are contained in $V$.

In this paper we extend the analysis to handle the $b$-constant.  We first show that to understand all subvarieties of $X$ with larger $a,b$ constants, only finitely many families of subvarieties need to be considered.  The following theorem shows that for any subvariety $Y$ lying outside of a fixed closed subset of $X$, either the $a,b$-values for $Y$ are smaller than those for $X$ or $Y$ is covered by subvarieties lying in a fixed bounded family.  This is the optimal statement, since such $Y$ need not themselves form a bounded family.

\begin{theo} \label{firstfinitefamily}
Let $X$ be a smooth uniruled projective variety of dimension $n$ over an algebraically closed field of characteristic $0$ and let $L$ be a big and semiample $\mathbb{Q}$-divisor on $X$.  There is a proper closed set $V \subset X$ and a finite collection of families of subvarieties $p_{i}: \mathcal{U}_{i} \to W_{i}$ with the evaluation map $s: \mathcal{U}_i \rightarrow X$ satisfying the following: for any subvariety $Y$ of $X$, either
\begin{enumerate}
\item $Y \subset V$,
\item $(a(Y,L),b(Y,L)) < (a(X,L),b(X,L))$ in the lexicographic order, or
\item $Y$ is the image of the closure of the locus $p_{i}^{-1}(T)$ by $s$ for some $T \subset W_{i}$, and the general fiber $Z$ over $T$ satisfies $(a(Y,L),b(Y,L)) \leq (a(Z,L),b(Z,L))$ in the lexicographic order.
\end{enumerate}
Furthermore, if $Z'$ denotes a resolution of the general fiber $Z$ of $p_{i}$ and $s': Z' \to X$ is the corresponding map, we have that $\kappa(K_{Z'} + a(Z', s'^*L)s'^{*}L|_{Z'}) = 0$.
\end{theo}

In particular, the pairs $(a(Y,L),b(Y,L))$ achieve a maximum in the lexicographic order as $Y$ varies over all subvarieties of $X$ not contained in $V$.  

The families of Theorem \ref{firstfinitefamily} have interesting properties.  We prove that for each such family the map $\mathcal{U}_{i} \to X$ is generically finite.  Furthermore, the family satisfies a basic dichotomy: either the map $\mathcal{U}_{i} \to X$ has degree $\geq 2$, or over an open subset of the parameter space $W_{i}$ there is a monodromy action on the local system $\mathcal{GN}$ describing relative N\'eron-Severi groups as in \cite{KM92}.

When we work over a number field, the dichotomy above yields thinness of point contributions: either the map is generically finite of degree $\geq 2$, or we can leverage the monodromy action using the Hilbert Irreducibility Theorem.  Combining everything, we obtain:

\begin{theo}
\label{theo: mainI}
Let $X$ be a geometrically uniruled smooth projective variety of dimension $n$ defined over a number field $F$ and let $L$ be a big and semiample $\mathbb{Q}$-divisor on $X$. Suppose that $\rho(X) = \rho(\overline{X})$ and $a(X, L)L + K_X$ is rigid.  As we vary over all geometrically integral subvarieties $Y \subset X$ defined over $F$ such that either $L|_{Y}$ is not big or
\begin{equation*}
(a(Y,L|_{Y}),b(F, Y,L|_{Y})) \geq (a(X,L),b(F, X,L))
\end{equation*}
in the lexicographic order, the points
\begin{equation*}
\bigcup_{Y} Y(F)
\end{equation*}
are contained in a thin subset of $X(F)$.

\end{theo}

\begin{rema}
The two geometric conditions in Theorem \ref{theo: mainI} are quite natural in our situation.  First, the restriction on the Picard group ensures that the $b$-value does not change when we pass to $\overline{F}$ to apply the minimal model program.  Second, when $\kappa(K_{Y} + a(Y, f^*L)f^{*}L) > 0$ then by \cite{LTT14} the fibers of the canonical map have $a,b$-values at least as large as $X$.  In this situation it makes sense to first study Conjecture \ref{mainconj} on the fibers, and then to understand the Iitaka fibration as a last step.
\end{rema}

\subsection{Generically finite maps}
We next turn to the geometric consistency of Manin's Conjecture with respect to generically finite maps.  The key conjecture is:

\begin{conj} \label{genfinconj1}
Let $X$ be a smooth uniruled projective variety over an algebraically closed field of characteristic $0$ and let $L$ be a big and nef $\mathbb{Q}$-divisor on $X$.  Up to birational equivalence, there are only finitely many generically finite covers $f: Y \to X$ such that $a(Y,f^{*}L) = a(X,L)$ and $\kappa(K_{Y} + a(Y, f^*L)f^{*}L) = 0$.
\end{conj}

We note that by the recent results of \cite{birkar16} the degree of the maps in Conjecture \ref{genfinconj1} is bounded, giving a weaker finiteness statement.  The condition $\kappa(K_{Y} + a(Y, f^*L)f^{*}L) = 0$ is a technical condition necessary for finiteness to hold.  When this condition fails, $Y$ (and hence $X$) is covered by subvarieties with higher $a,b$-values, and one can instead study the geometry of these subvarieties.  

Our approach to proving Conjecture \ref{genfinconj1} is to reduce to the finiteness of the \'etale fundamental group of $\mathbb{Q}$-Fano varieties as in  \cite{Xu14}, \cite{GKP16}.  We conjecture that any generically finite morphism as in Conjecture \ref{genfinconj1} is birationally equivalent to a morphism of $\mathbb{Q}$-Fano varieties that is \'etale in codimension $1$.  We are able to solve this problem in dimension $2$ and obtain partial results in dimension 3:

\begin{theo}
\label{theo: surfaces}
Let $S$ be a smooth projective surface over an algebraically closed field of characteristic $0$ and let $L$ be a big and nef $\mathbb{Q}$-divisor on $S$.  Then any generically finite cover $f: Y \to S$ such that $a(Y,f^{*}L) = a(S,L)$ and $\kappa(K_{Y} + a(S,L)f^{*}L) = 0$ is birationally equivalent to a cover which is \'etale in codimension $1$, and there are only finitely many such covers up to birational equivalence.

Furthermore if $\deg(f) \geq 2$ then $b(Y,f^{*}L) < b(S,L)$.
\end{theo}

\begin{theo}
\label{theo: Fano3folds}
Let $X$ be a smooth Fano threefold over an algebraically closed field of characteristic $0$ such that $X$ has index $\geq 2$ or has Picard rank $1$, index $1$ and is general in the moduli.  Then there is no dominant generically finite map $f: Y \to X$ of degree $\geq 2$ such that $\kappa(Y,-f^{*}K_{X} + K_Y)=0$.
\end{theo}

The precise analogue of Conjecture \ref{genfinconj1} for number fields does not hold due to the presence of twists.  However, we can still hope to show thinness of point contributions.  Due to the finiteness in Conjecture \ref{genfinconj1} over an algebraic closure, it suffices to consider the behavior of rational points for all twists of a fixed map. 

\begin{theo} \label{thinnessfortwists}
Let $X$ be a smooth projective variety over a number field $F$ satisfying $\rho(\overline{X})= \rho(X)$ and let $L$ be a big and nef $\mathbb Q$-divisor on $X$.  Suppose that $f : Y \rightarrow X$ is a generically finite $F$-cover from a smooth projective variety $Y$ satisfying $\kappa(K_{Y} + a(X, L)f^{*}L) = 0$.  As we let $\sigma$ vary over all elements of $H^1(F, \mathrm{Aut}(Y/X))$ such that the corresponding twists $f^{\sigma}: Y^{\sigma} \to X$ satisfy
\[
(a(Y, f^*L), b(F, Y^\sigma, (f^\sigma)^*L)) > (a(X, L), b(F, X, L)) 
\]
in the lexicographic order, the set
\[
\bigcup_{\sigma} f^\sigma(Y^\sigma (F)) \subset X(F)
\]
is contained in a thin subset of $X(F)$.
\end{theo}

In fact, when the induced extension of function fields $\overline{F}(\overline{Y})/\overline{F}(\overline{X})$ is not Galois, then the contributions of all twists lies in a thin set.  In the Galois case, we must leverage the comparison of $b$-values to apply the Hilbert Irreducibility Theorem.

\begin{rema}
As remarked above, Example \ref{exam: Sano} shows that the statement of Theorem \ref{thinnessfortwists} fails if we allow twists $\sigma$ such that
\begin{equation*}
(a(Y, f^*L), b(F, Y^\sigma, (f^\sigma)^*L)) = (a(X, L), b(F, X, L)).
\end{equation*}
Nevertheless this fact is still compatible with Manin's conjecture and Peyre's constant as Example \ref{exam: Sano} demonstrates.
\end{rema}

The paper is organized as follows. In Section~\ref{sec: balanced}, we recall the geometric invariants appearing in Manin's conjecture and the notion of balanced divisors. In Section~\ref{sec: preliminaries}, we collect several results from the literature on the minimal model program and birational geometry. Then we turn to families of subvarieties and prove Theorem~\ref{firstfinitefamily} and Theorem~\ref{theo: mainI} in Section~\ref{sec: families} and Section~\ref{sec: thinsets}. In Section~\ref{surfacesec}, we study generically finite covers for surfaces and prove Theorem~\ref{theo: surfaces}. In Section~\ref{fanothreefoldsec}, we study generically finite covers for Fano $3$-folds and prove Theorem~\ref{theo: Fano3folds}.
In Section~\ref{sec: twists}, we study contributions of twists and prove Theorem~\ref{thinnessfortwists}. In Section~\ref{examplesec}, we explore several examples to illustrate our study.

\noindent
{\bf Acknowledgments.}
The authors would like to thank I.~Cheltsov and C. Jiang for answering our questions about Fujita invariants and Fano threefolds.
They also would like to thank B. Hassett, D. Loughran, and Y. Tschinkel for answering our questions about twists. Finally they would like to thank M. Kawakita for his email communications regarding lengths of divisorial contractions on threefolds.
We thank D. Loughran, J. M\textsuperscript{c}Kernan, and M. Pieropan for useful comments.
Lehmann is supported by an NSA Young Investigator Grant.  Tanimoto is supported by Lars Hesselholt's Niels Bohr professorship.

\section{Balanced divisors}
\label{sec: balanced}

Here we recall some basic definitions of geometric invariants appearing in Manin's conjecture,
and the notion of balanced divisors.

\subsection{Geometric invariants}

Here we assume that our ground field $F$ is a field of characteristic zero,
but not necessarily algebraically closed.
In this paper, a variety defined over $F$ means a geometrically integral separated scheme of finite type over $k$.

Suppose that $X$ is a smooth projective variety defined over $F$.
We denote the N\'eron-Severi space on $X$ by $\mathrm{NS}(X)$
and its rank by $\rho(X)$.
If $\overline{X}$ is a base extension to an algebraic closure $\overline{F}$,
then in general, we have $\rho(X) \leq \rho(\overline{X})$,
and the strict inequality can happen.
We define the cone of pseudo-effective divisors as the closure of the cone of effective $\mathbb Q$-divisors in $\mathrm{NS}(X)$ and we denote it by $\Lambda_{\mathrm{eff}}(X)$.
\begin{defi}
Let $X$ be a smooth projective variety defined over $F$ 
and $L$ a big Cartier $\mathbb Q$-divisor on $X$.
The Fujita invariant is
\[
a(X, L) = \min \{ t \in \mathbb R : t[L] + [K_X] \in \Lambda_{\mathrm{eff}(X)}\}.
\]
For a singular projective variety $X$, we define $a(X, L)$ to be the Fujita invariant of the pullback of $L$ to a smooth model. This is well-defined because the invariant does not change upon pulling back $L$ via a birational morphism \cite[Proposition 2.7]{HTT15}.
Note that since cohomology of line bundles is invariant under flat base change,
we have $a(X, L) = a(\overline{X}, L)$.
Also by \cite{BDPP}, $a(X, L) >0$ if and only if $X$ is geometrically uniruled.
\end{defi}

For the purposes of Manin's Conjecture, it is often useful to add an assumption on the adjoint divisor $K_{X} + a(X,L)L$.

\begin{defi}
Let $(X, L)$ be a pair of a projective variety 
and a big and nef Cartier $\mathbb Q$-divisor defined over $F$.
We say that the adjoint divisor for $L$ is rigid, or that $(X, L)$ is adjoint rigid, if there exists a smooth resolution $\beta: \tilde{X} \rightarrow X$
such that $a(X, L)\beta^*L + K_{\tilde{X}}$ is rigid, i.e.,
the Iitaka dimension of $a(X, L)\beta^*L + K_{\tilde{X}}$ is zero.
Again this definition does not depend on the choice of the resolution $\beta: \tilde{X} \rightarrow X$.
\end{defi}

\begin{defi}
Let $X$ be a smooth projective variety defined over $F$ and $L$ a big Cartier $\mathbb Q$-divisor on $X$. We define the $b$-invariant by
\begin{align*}
b(F,X,L) = & \textrm{ the codimension of the minimal} \\
& \textrm{supported face of }  \Lambda_{\mathrm{eff}}(X) \textrm{ containing} \\
& \textrm{the numerical class } a(X, L)[L] + [K_X].
\end{align*}
Again this is a birational invariant upon pulling back $L$ \cite[Proposition 2.10]{HTT15},
so we define this invariant for singular projective varieties
via passage to a smooth model.
In general the $b$ constant is not preserved by base extension: we may have $b(F, X, L) \neq b(\overline{F}, \overline{X}, L) =: b(\overline{X}, L)$.
\end{defi}

\subsection{Balanced divisors}

Here we assume that our ground field is an algebraically closed field of characteristic zero.
The following notion is introduced in \cite[Definition 3.1]{HTT15}
\begin{defi}
Let $X$ be a uniruled projective variety and $L$ a big Cartier $\mathbb Q$-divisor on $X$.
Suppose that we have a morphism $f: Y \rightarrow X$ from a projective variety $Y$ which is generically finite onto its image.
Then $L$ is weakly balanced with respect to $f: Y \rightarrow X$  if
\begin{itemize}
\item $f^*L$ is big, and;
\item we have the inequality 
\[(a(Y, f^*L), b(Y, f^*L) ) \leq (a(X, L), b(X, L))\]
in the lexicographic order.
\end{itemize}
When the strict inequality holds, we say that $L$ is balanced with respect to $f: Y \rightarrow X$.

We say that $L$ is (weakly) balanced with respect to general subvarieties if there exists a proper closed subset $V$ such that $L$ is (weakly) balanced with respect to any subvariety $Y$ not contained in $V$.

We say that $L$ is (weakly) balanced with respect to generically finite covers if $L$ is (weakly) balanced with respect to every surjective generically finite morphism $f: Y \to X$.

We say that $L$ is (weakly) balanced with respect to general morphisms if there exists a proper closed subset $V$ such that $L$ is (weakly) balanced with respect to any generically finite morphism  $f: Y\rightarrow X$ whose image is not contained in $V$.
\end{defi}

It is also convenient to consider the following notion, which restricts attention to the $a$-values:

\begin{defi}
Let $X$ be a uniruled projective variety and $L$ a big Cartier $\mathbb Q$-divisor on $X$.
Suppose that we have a generically finite morphism $f: Y \rightarrow X$ onto its image.
Then $L$ is $a$-balanced (respectively strongly $a$-balanced) with respect to $f: Y \rightarrow X$  if
\begin{itemize}
\item $f^*L$ is big, and;
\item we have the inequality :
\[a(Y, f^*L)\leq a(X, L) \quad (\textrm{resp.~ }a(Y, f^*L) < a(X, L)).\]
\end{itemize}

We say that $L$ is (strongly) $a$-balanced with respect to general subvarieties if there exists a proper closed subset $V$ such that $L$ is (strongly) $a$-balanced with respect to any subvariety $Y$ not contained in $V$.
Other notions for generically finite morphisms and covers are defined in a similar way.
\end{defi}

\section{Preliminaries}
\label{sec: preliminaries}

We collect some useful results for our analysis.  In this section we work over an algebraically closed field of characteristic zero.

\subsection{The Minimal Model Program}

\begin{defi}
 A pair $(X, \Delta)$ is a normal $\mathbb Q$-factorial projective variety $X$ with an effective $\mathbb Q$-divisor $\Delta$.
 A pair $(X, \Delta)$ is called a $\mathbb Q$-factorial terminal log Fano pair if $(X, \Delta)$ has only terminal singularities
 and $-(K_X +\Delta)$ is ample.
\end{defi}

\begin{theo}[relative version of Wilson's theorem] \label{relativewilson}
 Let $f : X \rightarrow Y$ be a projective morphism between irreducible varieties and 
 $L$ a $f$-big and $f$-nef $\mathbb Q$-Cartier divisor on $X$. Then there exists an effective divisor $E$ such that
 for any sufficiently small $t$, $L-tE$ is $f$-ample.
\end{theo}
\begin{proof}
There is an open neighborhood of $L$ in $N^{1}(X/Y)$ consisting of $f$-big divisors.  In particular, there is some $f$-big divisor $L'$ such that $L - tL'$ is $f$-ample for any sufficiently small $t$.  If we replace $L'$ by $L' + \pi^{*}A$ for some sufficiently ample divisor $A$ on $Y$, then we can find an effective divisor $E$ that is $\mathbb{Q}$-linearly equivalent to $L'$ without losing the desired property.
\end{proof}

Let $f : X \rightarrow Y$ be a smooth projective morphism between irreducible varieties. 
According to \cite[Proposition 12.2.5]{KM92}, there exists a local system $\mathcal GN^1(X/Y)$ over $Y$ in the analytic topology (or \'etale topology)
with a finite monodromy. 
When $H^1(X_t, \mathcal O_{X_t}) = 0$ for any $t \in Y$, we have an isomorphism $\mathcal GN^1(X/Y)|_t \cong N^1(X_t)$ for any $t \in Y$ by Hodge theory.

\begin{rema} \label{kmremark}
The two results above in \cite{KM92} are stated over $\mathbb C$ in the analytic topology, but one can prove the same results over any algebraically closed field of characteristic zero in the \'etale topology by a comparison theorem.

First suppose that we have a smooth projective morphism $X/S$ defined over an algebraically closed subfield $F \subset \mathbb{C}$
such that any fiber $X_s$ satisfies $H^1(X_s, \mathcal O_{X_s}) =0$.
Consider the following functor:
\[
\Phi: [S' \rightarrow S]  \mapsto N^1(X\times_S S'/S').
\]
It follows from \cite[Proposition 12.2.3]{KM92} that this functor is represented by a proper separated unramified algebraic space $\mathcal H/S$. Then we can define the sheaf $\mathcal GN^1(X/S)$ of sections of $\mathcal H/S$ with open support in the \'etale topology.
We claim that this satisfies the properties we stated above. Indeed, after base change to $\mathbb C$ and considering the sheaf in the analytic topology, $\mathcal GN^1(X/S)$ forms a local system with finite monodromy action. Then we have a finite \'etale cover $T \rightarrow S$ over $\mathbb C$ such that the local system $\mathcal GN^1(X/S)$ is trivialized by $T$. This means that $\mathcal GN^1(X/S)$ is a local system in \'etale topology, and stalks in the analytic topology and \'etale topology coincide.
Finally algebraic fundamental groups and N\'eron-Severi groups are invariant under base change of algebraically closed fields, so our assertion follows from results over $\mathbb C$ in the analytic topology.

For an arbitrary algebraically closed field $F$ of characteristic zero, a standard argument in algebraic geometry shows that our family is defined over a subfield $F' \subset F$ which admits an embedding into $\mathbb C$, so our assertion again follows from a comparison theorem.
\end{rema}

\begin{theo}[the relative MMP vs the absolute MMP]
\label{theo: relativeMMP}
 Let $f : X \rightarrow Y$ be a smooth projective morphism from a smooth irreducible variety to an irreducible variety and 
 $L$ a $f$-big and $f$-nef $\mathbb Q$-Cartier divisor on $X$. 
 Assume that $H^1(X_t, \mathcal O_{X_t}) = 0$ for any $t \in Y$
 and the monodromy action on $\mathcal GN^1(X/Y)$ is trivial.
 Consider the relative $L + K_X$-MMP over $Y$
 \[
  X = X_0 \dashrightarrow X_1 \dashrightarrow \cdots \dashrightarrow X_n.
 \]
 Then for a general $t \in Y$, $X_{0,t}\dashrightarrow \cdots \dashrightarrow X_{n,t}$ is the absolute $L|_{X_t}+K_{X_t}$-MMP.
 If $X_{i,t} \dashrightarrow X_{i+1,t}$ is not an isomorphhism, then we have
 \begin{itemize}
  \item the map $X_{i,t} \dashrightarrow X_{i+1,t}$ is a divisorial contraction if $X_i \dashrightarrow X_{i+1}$ is;
  \item the map $X_{i,t} \dashrightarrow X_{i+1,t}$ is a flip if $X_i \dashrightarrow X_{i+1}$ is;
  \item the map $X_{i,t} \dashrightarrow X_{i+1,t}$ is a Mori fiber space if $X_i \dashrightarrow X_{i+1}$ is.
 \end{itemize}
.

\end{theo}
\begin{proof}
This follows from \cite[Theorem 4.1]{dFH09} and its proof.
\end{proof}

\begin{theo}
\label{theo: MMPtoweakFano}
 Let $X$ be a smooth projective variety and $L$ a big and nef $\mathbb Q$-divisor on $X$.
 Suppose that $(X,L)$ is adjoint rigid. Then there is a birational contraction $\phi : X \dashrightarrow X'$ 
 to a $\mathbb Q$-factorial terminal weak Fano variety $X'$ such that
 \[
  a(X, L)\phi_*L + K_{X'} \equiv 0, \quad b(X, L) = b(X', \phi_*L).
 \]
\end{theo}
\begin{proof}
 We apply $(a(X, L)L+K_X)$-MMP and obtain a birational contraction $\psi : X \dashrightarrow \tilde{X}$ such that $a(X, L)\psi_*L +K_{\tilde{X}} \equiv 0$.
 By \cite[Lemma 3.5]{LTT14}, we have $b(X, L) = b(\tilde{X}, \psi_*L) = \mathrm{rk} \, \mathrm{NS}(\tilde{X})$.
 As explained in \cite[Lemma 2.4 and Proposition 2.5]{LTT14}, $\tilde{X}$ is a $\mathbb Q$-factorial terminal log Fano variety.
 Hence it follows from \cite{BCHM} that $\tilde{X}$ is a Mori dream space.
 Since $L$ is nef, the pushforward $\psi_*L$ is in the movable cone.
 Thus we can find a small $\mathbb{Q}$-factorial modification $f : \tilde{X} \dashrightarrow X'$  
 such that $a(X, L)f_*\psi_*L \equiv -K_{Ẍ́'}$ is big and nef.
 We claim that $X'$ has only terminal singularities.
 
 Indeed, as explained in \cite[Lemma 2.4]{LTT14}, we can find an effective $\mathbb Q$-divisor $L' \equiv L$ such that
 the pair $(\tilde{X}, a(X, L)\psi_*L')$ is a terminal pair.
 Let $H$ be an ample effective divisor on $X'$ and we denote its strict transform on $\tilde{X}$ by $H'$.
 Then for a sufficiently small $\epsilon >0$, we still have that $(\tilde{X}, \epsilon H' + a(X, L)\psi_*L')$ is a terminal pair.
 We apply $(\epsilon H' + a(X, L)\psi_*L' + K_{\tilde{X}})$-MMP to $\tilde{X}$.
 The result must be $f: X \dashrightarrow X'$ since $H$ is ample and $X'$ is $\mathbb Q$-factorial.
 Since $f: \tilde{X} \dashrightarrow X'$ is an isomorphism in codimension one, $f$ decomposes into a composition of flips.
 We conclude that $(X', \epsilon H + a(X, L)f_*\psi_*L')$ is a terminal pair.
 It is easy to see that 
 \[
  b(X, L) = \mathrm{rk} \, \mathrm{NS}(\tilde{X}) = \mathrm{rk} \, \mathrm{NS}(X') = b(X', f_*\psi_*L).
 \]
Thus our assertion follows.
\end{proof}

\subsection{Explicit Fujita-type statements}

To understand the balanced property for example, it is crucial to have effective Fujita-type results.  One approach is to use the work of \cite{Reider} for surfaces and its extensions to dimension $3$ (see for example \cite{Lee99}).  However, we are usually forced to work with $\mathbb{Q}$-divisors, so it will be more useful to rely on effective volume bounds for singular Fano varieties.

\begin{prop} \label{threefoldbigbound}
Let $X$ be a smooth projective variety and let $L$ be a big and nef $\mathbb{Q}$-divisor on $X$.
\begin{enumerate}
\item Suppose $\dim X = 1$.  If $\vol(L)>2$ then $K_{X} + L$ is big.
\item Suppose $\dim X = 2$.  If $\vol(L)>9$ and $L \cdot C > 2$ for every curve through a general point of $X$ then $K_{X} + L$ is big.  
Furthermore, if $a(X, L)L+K_X$ is rigid, then we have $\vol(L) \leq 9/a(X,L)^2$.
\item Suppose $\dim X = 3$.  
If $\vol(L) > 64$ and $L^{2} \cdot S > 9$ for every surface through a general point 
and $L \cdot C > 2$ for every curve through a general point then $K_{X} + L$ is big. 
\end{enumerate}
\end{prop}

\begin{proof}
In each case we need to show that if $L$ satisfies the given criteria then $a(X,L)<1$.  The proof is by induction on the dimension.  
We focus on (3), since the cases (1), (2) use exactly the same argument but are easier.

Suppose that $a(X, L)X+K_X$ is not rigid. We run the $a(X, L)L +K_X$-MMP and obtain the minimal model $\phi : X \dashrightarrow \tilde{X}$.
Let $f: \tilde{X} \rightarrow Y$ be the semiample fibration associated to $a(X, L)\phi_*L + K_{\tilde{X}}$.
After applying a resolution of indeterminacy, we may assume that $\phi$ is a birational morphism.
Let $F$ be a general fiber of $f\circ \phi$. Then we have $a(X, L) = a(F, L|_F)$. It follows from the theorem in lower dimensions that $a(F, L|_F) < 1$.

Assume that $a(X, L)X+K_X$ is rigid.
Again we apply the $a(X, L)L +K_X$-MMP with scaling of $L$ and obtain the minimal model $\phi : X \dashrightarrow \tilde{X}$.
Fix a small $\epsilon > 0$ and continue to run the $(K_{X} + (a(X,L)-\epsilon)L)$-MMP with scaling of $L$.  
The result is a $\mathbb Q$-factorial terminal threefold $\psi: X \dashrightarrow X'$ admitting a Mori fibration $\pi: X' \to Z$ 
whose general fiber is a terminal Fano variety such that $K_{X'} + a(X,L)\psi_{*}L$ is trivial along the fibers.

If the fibers of $\pi$ have dimension $1$ or $2$, then by resolving $\psi$ we may as well suppose that the rational map $\psi$ is a morphism.  
Then $a(X,L) = a(F,L)$ for a general fiber $F$ of $\pi \circ \psi$. By the theorem in lower dimensions, we have $a(X, L) <1$.

When $Z$ is a point, $X'$ is a terminal $\mathbb{Q}$-Fano of Picard rank $1$.
In \cite{Nami}, Namikawa showed that every terminal Gorenstein Fano 3-fold $Y$ can be deformed to a smooth Fano 3-fold, 
in particular we have $\vol (-K_Y) \leq 64$.
Moreover Prokhorov proved that the degree of a $\mathbb Q$-factorial terminal non-Gorenstein Fano 3-fold of Picard rank one 
is bounded by $125/2$ (\cite{Pro}).
All together, we have $\vol(-K_{X'}) \leq 64$.
Since we have $\vol(L) \leq \vol(\psi^*\psi_*L)= \vol(\psi_*L)$, we conclude that $a(X, L) <1$.
\end{proof}

We can phrase the result to look more like Reider's theorem:

\begin{theo} \label{theo: surfacebigbound}
Let $X$ be a smooth surface and let $L$ be a big and nef $\mathbb{Q}$-divisor on $X$.  
Suppose that $L \cdot C > 3$ for every rational curve $C$ passing through a general point.  Then $K_{X} + L$ is big.
\end{theo}

\begin{proof}
Let $\pi: X \to \tilde{X}$ be a minimal model for $K_{X} + a(X,L)L$.  
By the curve condition, we only need to consider the case when $K_{X} + a(X,L)L$ is rigid.
We continue the $K_{X} + (a(X,L)-\epsilon)L$-MMP, and obtain $\phi : \tilde{X} \rightarrow X'$ with a Mori fiber structure $f : X' \rightarrow Z$.
If $\dim Z = 1$, then our curve condition implies that $a(X, L) <1$.
Hence we assume that $X'$ is a smooth del Pezzo surface of Picard rank one, i.e., $\mathbb P^2$ and $a(X,L)\phi_*\pi_*L \equiv -K_{X'}$.

There is an ample rational curve on $X'$ satisfying $-K_{X'} \cdot C = 3$ 
-- one can simply take a general line on $\mathbb{P}^{2}$.  
Since the classes of $\phi_*\pi_*L$ and $-K_{X'}$ are proportional, by our condition on curves we see that $a(X,L)<1$.
\end{proof}

\begin{coro}
\label{threefoldbigboundimproved}
 Let $X$ be a smooth projective 3-fold and let $L$ be a big and nef $\mathbb{Q}$-divisor on $X$.
 If $\vol(L)>64$ and $L \cdot C > 3$ for every rational curve through a general point then $K_{X} + L$ is big.
\end{coro}

\subsection{Rationally connected varieties}
By an argument of \cite[II.5.15 Lemma]{Nak}, the kernel in the following lemma does not depend on the choice of general fiber $F$.  

\begin{lemm} \label{rcfibers}
Let $X$ be a normal $\mathbb{Q}$-factorial variety and suppose that $\pi: X \to Z$ is a morphism whose general fiber is irreducible and rationally connected.  Let $K$ denote the kernel of the restriction map $N^{1}(X) \to N^{1}(F)$ for a general fiber $F$.  Then $K$ is spanned by a finite collection of effective irreducible $\pi$-vertical divisors.
\end{lemm}

\begin{proof}
See the proof of \cite[Theorem 4.5]{LTT14}.
\end{proof}

\section{Families of varieties}
\label{sec: families}

In this section we address the geometric behavior of the $a,b$-constants for families of subvarieties.  The following proposition is useful for us.

\begin{prop}
 Let $f : X \rightarrow Y$ be a morphism of finite type from a noetherian scheme to an irreducible noetherian scheme. Then
 \[
  E=\{ y \in Y | \text{$X_y$ is geometrically integral}\}
 \]
 is constructible in $Y$. In particular, if the generic point $\eta$ of $Y$ is contained in $E$,
 then there exists a nonempty open set $U \subset Y$ such that $U \subset E$.
\end{prop}
\begin{proof}
 See \cite{stacks} Lemma 36.22.7., Lemma 36.21.5., and Lemma 27.2.2.
\end{proof}

\begin{defi}
 A morphism $p : U \rightarrow W$ between irreducible varieties is called {\it a family of varieties} if the generic fiber is geometrically integral.
 A family of projective varieties is a projective morphism which is a family of varieties.  A family of subvarieties is a family of projective varieties admitting a morphism $s: U \to X$ which restricts to a closed immersion on every fiber of $p$. 
\end{defi}

From now on in this section we work over a fixed algebraically closed field of characteristic zero.

\subsection{Variation of constants in families}

We show that the geometric invariants are constant for general members of a family of projective varieties.

Let $X$ be a uniruled projective variety and $L$ a big Cartier divisor on $X$.
The adjoint divisor for $(X, L)$ is the $\mathbb R$-divisor $a(X, L)\beta^*L + K_{\tilde{X}}$ for
some smooth resolution $\beta: \tilde{X} \rightarrow X$. The Iitaka dimension of the adjoint divisor does not depend on the choice of $\beta$.

\begin{theo}[\cite{HMX13}] \label{invarianceofplurigenera}
 Let $f: X \rightarrow Y$ be a family of projective varieties. Suppose that $L$ is a $f$-big and $f$-nef Cartier divisor on $X$.
 Then there exists a nonempty open subset $U\subset Y$ such that the invariant $a(X_t, L|_{X_t})$ is constant for $t \in U$, and 
 the Iitaka dimension of the adjoint divisor for $(X_t, L|_{X_t})$ is constant for $t \in U$.
\end{theo}
\begin{proof}
By resolving and throwing away a closed subset of the base, we may suppose that every fiber is smooth.  
By Theorem \ref{relativewilson}, there is a fixed effective divisor $E$ such that for any sufficiently small $s$ we have $L \sim_{\mathbb{Q},Y} A+sE$ where $A$ is $f$-ample.
Fix a positive rational number $a$.
After a further blow-up resolving $E$ and after replacing $L$ by a $\mathbb{Q},Y$-linearly equivalent divisor,
we may suppose that $L$ has simple normal crossing support and $(X,aL)$ is log canonical.  
Furthermore, after throwing away a closed subset of the base we may suppose that every component of $L$ dominates $Y$.

We then apply the invariance of log plurigenera as in \cite[Theorem 1.8]{HMX13}. 
We conclude that for any sufficiently divisible $m$ and for every fiber $X_{t}$ the value of $h^{0}(X_{t},\mathcal{O}_{X_{t}}(m(K_{X_{t}}+aL|_{X_{t}})))$ is independent of $t$.  
The desired conclusion is immediate.
\end{proof}

\begin{prop} \label{bconstancy}
 Let $f: X \rightarrow Y$ be a family of projective varieties. Suppose that $L$ is a $f$-big and $f$-nef Cartier divisor on $X$.
 Assume that for a general member $X_t$, the adjoint divisor $a(X_t, L|_{X_t})L|_{X_t} + K_{X_t}$ is rigid.
 Then there exists a nonempty subset $U\subset Y$ such that $b(X_t, L|_{X_{t}})$ is constant over $U$.
\end{prop}

\begin{proof}
 By resolving and throwing away a closed subset of the base, we may assume that (i) every fiber $X_t$ is smooth for $t \in Y$, (ii) $a = a(X_t, L|_{X_t})$ is constant over $Y$, 
 (iii) $a(X_t, L|_{X_t})L|_{X_t} + K_{X_t}$ is rigid for every $t$.  It follows from \cite[Proposition 2.5]{LTT14} that every fiber is rationally connected,
 so we conclude that $H^1(X_t, \mathcal O_{X_t}) = 0$ for any $t \in Y$. 
 \cite[Proposition 12.2.5]{KM92} indicates that the sheaf $\mathcal GN^1(X/S)$ forms a local system in the \'etale topology (see Remark \ref{kmremark})
 with a finite monodromy action. Replacing $Y$ by a finite \'etale cover and taking a base change, we may assume that the monodromy action is trivial,
 and in particular, we have a natural isomorphism $N^1(X/Y) \cong N^1(X_t)$ for a general $t$, and dually $N_1(X/Y) \cong N_1(X_t)$.
 Then we apply the relative $aL +K_X$-MMP over $Y$ and obtain a birational contraction map $\psi : X \dashrightarrow X'$ to a relative minimal model over $Y$.
 It follows from Theorem \ref{theo: relativeMMP} that for a general $t$, $X'_t$ is a minimal model of $aL|_{X_t}+K_{X_t}$-MMP,
 hence \cite[Lemma 3.5]{LTT14} implies that $b(X_t, L|_{X_t}) = b(X'_t, \psi_*L|_{X'_t}) = \mathrm{rank} \, N^1(X'_t)$. This is constant for a general $t$ by combining Theorem \ref{theo: relativeMMP} with the constancy of $\mathrm{rank} \, N^1(X_t)$.
 \end{proof}
 
 \begin{rema}
It is interesting to ask whether the adjoint rigidity of $(X_{t},L|_{X_{t}})$ is necessary in Proposition \ref{bconstancy}.  If we make no assumption on $X_{t}$, then the most we can ask for is constancy of the $b$-value for a very general fiber (consider for example a family of K3 surfaces where the Picard rank jumps infinitely often), and it should be possible to prove such a statement using invariance of plurigenera.  However, if $X_{t}$ is uniruled, then it seems likely that the $b$-value should again be constant on a general fiber.
 \end{rema}

\subsection{Universal families of subvarieties breaking the balanced property}

In this section, we construct the universal families of subvarieties breaking the balanced property.
Let us recall some results from \cite{LTT14} and \cite{HJ16}.

\begin{lemm}
\label{lemm: canonical}
 Let $X$ be a smooth uniruled variety and $L$ a big and nef divisor on $X$.
 Suppose that the adjoint divisor $a(X, L)L +K_X$ is not rigid.
 Let $\phi : X \dashrightarrow Z$ be the canonical fibration associated to $a(X, L)L +K_X$.
 Let $Y$ be a general fiber of $\phi$. Then we have
 \[
  a(X, L) = a(Y,L), \quad b(X, L) \leq b(Y, L).
 \]
 In particular, $L$ is not balanced with respect to $Y$.
\end{lemm}
\begin{proof}
 See \cite[Theorem 4.5]{LTT14} and its proof.
\end{proof}

\begin{theo}[\cite{HJ16} Corollary 2.15]
Let $X$ be a smooth projective variety of dimension $n$ and let $L$ be a semiample big $\mathbb{Q}$-divisor on $X$.  Fix a positive real number $\tau$.  As we vary $Y$ over all subvarieties of $X$, there are only finitely many values of $a(Y,L)$ which are at least $\tau$.
\end{theo}

\begin{proof}
We may rescale $L$ to assume it is Cartier.  By an argument of Siu, $K_{Y} + nL|_{Y}$ is pseudo-effective for any resolution $Y$ of a subvariety of $X$.  Thus $a(Y,L)$ is bounded above by $n$ as we vary over all subvarieties $Y$.  By replacing $L$ by $2nL$, we may ensure that $a(Y,L) \leq \frac{1}{2}$.  We may then apply \cite[Corollary 2.15]{HJ16} to an appropriate rescaling of $\tau$.
\end{proof}

\begin{coro} \label{coro: bounded}
Let $X$ be a projective variety of dimension $n$ and let $L$ be a semiample big $\mathbb{Q}$-divisor on $X$.  Fix a positive real number $\tau$.  The set of all subvarieties $Y$ that are adjoint rigid, not contained in $\mathbf{B}_{+}(L)$, and satisfy $a(Y,L) \geq \tau$ are parametrized by a bounded subset of $\Chow(X)$.
\end{coro}

\begin{proof}
By \cite[Corollary 2.15]{HJ16} we may suppose after rescaling $L$ that $L$ is Cartier, $|L|$ is basepoint free, and $a(Y,L) < 1$ for every subvariety $Y$ of $X$. Our goal is to show that the $L$-degree of $Y$ is bounded.   Thus, we may apply a resolution of singularities so that $X$ and $Y$ are smooth.  

 Applying the $K_{Y} + a(Y,L)L|_{Y}$-MMP, we obtain a birational contraction $\phi : Y \dashrightarrow Y'$ 
 to a $\mathbb Q$-factorial terminal weak Fano variety $Y'$. (See Theorem~\ref{theo: MMPtoweakFano}.)
 Note that $a(Y, L)\phi_*L|_{Y} +K_{Y'} \equiv 0$.  Furthermore, the coefficients of the divisor $a(Y,L)\phi_{*}L$ can only attain a finite number of values.  \cite[Theorem 1.3]{HX14} shows that the pairs $(Y',a(Y,L)\phi_{*}L|_{Y})$ are parametrized by a bounded family.  We deduce that there are only finitely many possible values of $\vol(-K_{Y'})$ as we vary $Y$ using invariance of log plurigenera.  Then
 \begin{equation*}
 L^{\dim Y} \cdot Y = \mathrm{Vol(L|_{Y})}\leq \mathrm{Vol(\phi_*L|_{Y})} = \frac{1}{a(Y,L)^{\dim Y}} \vol(-K_{Y'})
 \end{equation*}
and our assertion follows.
\end{proof}

Using this corollary and arguing as in \cite[Theorem 4.8]{LTT14} one obtains:

\begin{theo}[\cite{HJ16} Theorem 1] \label{thm: a-exceptional}
Let $X$ be a smooth uniruled projective variety and let $L$ be a semiample big $\mathbb{Q}$-divisor on $X$.  There is a closed subset $V \subset X$ such that any subvariety $Y$ with $a(Y,L) > a(X,L)$ is contained in $V$.
\end{theo}

Next we define a family of subvarieties breaking the balanced property.

\begin{defi} \label{defi: breaking balanced}
 Let $X$ be a projective variety and $L$ a big and nef $\mathbb Q$-divisor on $X$.
 A family of subvarieties $\pi :\mathcal U \rightarrow W$ is called {\it a family of subvarieties breaking the balanced property}
 if 
 \begin{enumerate}
  \item the evaluation map $s : \mathcal U \rightarrow X$ is dominant,
  \item $W$ is normal,
  \item for a general member $Y$ of $\pi$, we have $a(Y, L|_Y) = a(X, L)$,
  \item for a general member $Y$ of $\pi$, the adjoint divisor for $L|_Y$ is rigid, and 
  \item for a general member $Y$ of $\pi$, $b(X, L) \leq b(Y, L|_Y)$.
 \end{enumerate}
 We say that a family of varieties $\pi: \mathcal{U} \rightarrow W$ {\it breaks the balanced property} if the same conditions hold, except
 we only require that $s$ restricted to every fiber of $\pi$ is generically finite onto its image.
\end{defi}

\begin{theo}
\label{theo: universal families}
 Let $X$ be a smooth uniruled projective variety of dimension $n$ and $L$ a big and semiample $\mathbb Q$-divisor on $X$.
 Then there exists a proper closed subset $V \subset X$ and a finite number of families of subvarieties breaking the balanced property
 $\pi_i : \mathcal U_i \rightarrow W_i$ with evaluation maps $s_i : \mathcal U_i \rightarrow X$ such that for every subvariety $Y$ breaking the balanced property, either 
 \begin{itemize}
  \item $Y \subset V$ or,
  \item there exist $i$ and a subvariety $Z \subset W_i$ such that $Q_{Y} := \pi_i^{-1}(Z)$ satisfies $\overline{s_{i}(Q_{Y})} = Y$ and $s_{i}|_{Q_{Y}}$ is birational onto its image.  In this case the fibers of $\pi_{i}$ break the balanced property for $Y$.
 \end{itemize}
 We call $\pi_i : \mathcal U_i \rightarrow W_i$ the universal families of subvarieties breaking the balanced property.
\end{theo}

\begin{proof}
 Let $V$ be a proper closed subset as in the statement of Theorem~\ref{thm: a-exceptional}.
 By enlarging $V$ if necessary, we may assume that $V$ contains $\mathbf B_+(L)$.
 We consider the locus $W^\circ \subset \mathrm{Chow}(X)$ parametrizing irreducible and reduced subvarieties $Y$ 
 where $Y$ satisfies:
 \begin{itemize}
  \item $Y$ is not contained in $V$;
  \item the adjoint divisor for $L|_Y$ is rigid;
  \item we have $a(Y, L|_Y) = a(X, L)$ and $b(X, L) \leq b(Y, L|_Y)$.
 \end{itemize}
 It follows from 
 Corollary~\ref{coro: bounded}
 that $W^\circ$ is a bounded family. Let $W$ be the Zariski closure of $W^\circ$ in $\mathrm{Chow}(X)$.
 Let $W_i$ be irreducible components of $W$ and $\pi_i: \mathcal U_i \rightarrow W_i$ the pullback of the universal family on $\mathrm{Chow}(X)$.
 We denote the evaluation map for each family by $s_i : \mathcal U_i \rightarrow X$.
 If $s_i$ is not dominant, then we add the closure of the image of $s_i$ to $V$,
 so that we may assume that for any $i$, $\pi_i : \mathcal U_i \rightarrow W_i$ is a family of subvarieties breaking the balanced property
 (after taking the normalization of $W_i$ if necessary).
 
 We claim that the $V$ and $\pi_i : \mathcal U_i \rightarrow W_i$ we constructed satisfy the desired properties.
 Let $Y$ be a subvariety such that $Y$ is not contained in $V$ and $L$ is not balanced with respect to $Y$.
 If the adjoint divisor for $L|_Y$ is rigid, then it is easy to see that $Y$ is a member of $\pi_i : \mathcal U_i \rightarrow W_i$ for some $i$.
 Suppose that the adjoint divisor for $L|_Y$ is not rigid.
 Let $\beta : Y' \rightarrow Y$ be a smooth resolution and 
 $\phi : Y' \dashrightarrow Z$ be the canonical fibration associated to $a(Y, L|_Y)\beta^*L +K_{Y'}$.
 Let $\Gamma \subset Y \times Z$ be the graph of the rational map $\phi\circ \beta^{-1} : Y \dashrightarrow Z$.
 It follows from the universal property of $\mathrm{Chow}(X)$ that there exists a dense open set $Z_0 \subset Z$
 and a morphism $\psi : Z_0 \rightarrow \mathrm{Chow}(X)$ such that $\Gamma|_{Z_0}$ is the pullback of the universal family on $\mathrm{Chow}(X)$.
 Lemma~\ref{lemm: canonical} implies that this $\psi$ actually factors through $W_i$ for some $i$.
 We denote the image of $\psi$ by $Z'$. It follows from the construction that $Y$ is the closure of $s_i(\pi_i^{-1}(Z'))$ and that $s_{i}$ is birational on $\pi_{i}^{-1}(Z')$.
\end{proof}

\begin{rema}
 The universal families of subvarieties breaking the balanced property constructed in the proof of Theorem~\ref{theo: universal families}
 actually satisfy a stronger universal property. 
 Suppose that we have a family of subvarieties breaking the balanced property $\pi : \mathcal U \rightarrow W$.
 Then there exist a dense open subset $W^\circ \subset W$ and a morphism $\psi : W^\circ \rightarrow W_i$ for some $i$
 such that $\mathcal U|_{W^\circ}$ is isomorphic to $\psi^*\mathcal U_i$.
 In particular the families are unique up to birational isomorphisms.
\end{rema}

\begin{proof}[Proof of Theorem \ref{firstfinitefamily}:]
Choose $V$ and the families $p_{i}: \mathcal{U}_{i} \to W_{i}$ as the $\pi_{i}$ in the statement of Theorem \ref{theo: universal families}.  By Definition \ref{defi: breaking balanced}.(4), the general fiber $Z$ of $p_{i}$ satisfies the desired adjoint rigidity.  Suppose that $Y$ is a subvariety of $X$ which breaks the balanced property and is not contained in $V$.  According to Theorem \ref{theo: universal families}, there is a locus $T \subset W_{i}$ for some $i$ such that the closure of the $s_{i}$-image of $T$ is equal to $Y$, and the covering family of $p_{i}$-fibers breaks the balanced condition for $Y$.  This is exactly the desired statement.
\end{proof}

It follows that the $a,b$-constants achieve a maximum as we vary over all subvarieties.  This is of course a necessary condition for the geometric consistency of Manin's Conjecture.  More precisely:

\begin{coro}
Let $X$ be a smooth uniruled projective variety of dimension $n$ and let $L$ be a big and semiample $\mathbb{Q}$-divisor on $X$.  
As $Y$ varies over all subvarieties of $X$ not contained in $\mathbf{B}_{+}(L)$
the values $(a(Y,L),b(Y,L))$ achieve a maximum in the lexicographic order.
\end{coro}

\begin{proof}
Let $Y$ be any subvariety of $X$ not contained in $\mathbf{B}_{+}(L)$.  Applying Theorem \ref{theo: universal families}, it suffices to consider only those $Y$ such that $(Y,L|_{Y})$ is adjoint rigid.  By combining Corollary \ref{coro: bounded}, Theorem \ref{invarianceofplurigenera}, and Proposition \ref{bconstancy} we conclude the desired statement by Noetherian induction.
\end{proof}

\subsection{Geometric properties of families}

We next discuss the geometric results underlying thin set results for families of subvarieties.

\begin{prop} \label{prop: genfinofmaps}
Let $X$ be a smooth uniruled projective variety of dimension $n$ and let $L$ be a big and nef $\mathbb{Q}$-divisor on $X$ such that $(X,L)$ is adjoint rigid.  Suppose that $p: \mathcal{U} \to W$ is a family of 
subvarieties of dimension $d$ breaking the balanced condition via a dominant map $s: \mathcal{U} \to X$ such that the induced rational map $W \dashrightarrow \mathrm{Chow}(X)$ is birational onto its image.
Then $s$ is generically finite.
\end{prop}

\begin{proof}

Suppose that $s: \mathcal{U} \rightarrow X$ is not generically finite so that a general fiber is positive dimensional.
After resolving singularities of $\mathcal{U}$ and shrinking $W$, we may assume that our family $p: \mathcal{U} \rightarrow W$ is smooth.
After taking a finite \'etale base change of $W$, we may assume that the monodromy action on $\mathcal GN^1(\mathcal{U} /W)$ is trivial.
After shrinking $W$ if necessary, we assume that $W$ is smooth and the induced map $W \rightarrow \mathrm{Chow}(X)$ is quasi finite.
We apply the relative $a(X, L)s^*L + K_{\mathcal{U}}$-MMP and obtain a relative minimal model $\tilde{\mathcal{U}}$
and we denote the exceptional divisors of $\mathcal{U} \dashrightarrow \tilde{\mathcal{U}}$ by $E_1, \cdots, E_t$ and their union by $E$.
It follows from Theorem~\ref{theo: relativeMMP} that for a general member $Y$ of $p : \mathcal{U} \rightarrow W$, the support of the rigid divisor numerically equivalent to $a(X, L)s^{*}L + K_Y$ is contained in $E \cap Y$.
After shrinking $W$ if necessary, we may assume that this holds for any member $Y$ of $p$.

Choose a general complete intersection variety $V$ in $X$ of dimension $(n-d-1)$.
Then $s^{-1}(V)$ is a subvariety of $\mathcal{U}$ with dimension at least $(n-d)$, and by generality we know that the general point of $s^{-1}(V)$ is not contained in $E$.  Furthermore, we claim that
\begin{equation*}
\overline{s(p^{-1}(p(s^{-1}(V))))} = X.
\end{equation*}
Indeed, if we let $p': \mathcal{U} \times_{W} \mathcal{U} \to X$ be the composition of the first projection with $s$, and $s'$ be the composition of the second projection with $s$, then it suffices to show that $\overline{s'(p'^{-1}(V))} = X$.  Note that the fibers of $p'$ have dimension at least $d+1$, since there is at least a one-dimensional family of subvarieties parametrized by $W$ through each point of $X$.  Since $s'$ is dominant, its restriction to any general $(n-d-1)$-subvariety of the parameter space $X$ is also dominant.  Altogether we obtain the claim.

By cutting by hyperplanes, we can choose a general subvariety $W' \subset p(s^{-1}(V))$ of dimension $(n-d)$ such that
\begin{itemize}
\item the evaluation map $s : \mathcal{U}' = p^{-1}(W') \rightarrow X$ is dominant and generically finite;
\item and for the general fiber $Y'$ of $p: \mathcal{U}' \to W'$, the points $s^{-1}(V) \cap Y'$ are not contained in $E$.
\end{itemize}
After shrinking $W'$, we may assume that $W'$ and $\mathcal{U}'$ are smooth.  This implies that the ramification locus of $s: \mathcal{U}' \rightarrow X$ is divisorial.  Let $D$ be a rigid effective divisor on $X$ which is numerically equivalent to $a(X, L)L + K_X$.
Pick a general member $Y'$ of $p: \mathcal{U}' \rightarrow W'$ such that
\begin{itemize}
\item $Y'$ is not contained in $s^*D$
\item $Y'$ is not contained in the ramification locus of $s: \mathcal{U}' \rightarrow X$;
\item there is a point $y' \in s^{-1}(V) \cap Y'$ that is not contained in $E \cap Y'$.
\end{itemize}
We have $a(X, L)s^*L + K_{\mathcal{U}'} \equiv s^*D + R$ where $R$ the ramification divisor of $s: \mathcal{U}' \rightarrow X$.
Since $y'$ is contained in the exceptional locus of $s: \mathcal{U}' \rightarrow X$, we have $y' \in R$.  However, this is a contradiction to the fact that $y' \not \in E \cap Y'$.
\end{proof}

\begin{exam}
Proposition \ref{prop: genfinofmaps} has interesting feedback with the deformation theory of rational curves.  For example, suppose that $X$ admits a dominant family of rational curves breaking the balanced condition.  Then we have
\begin{equation*}
a(X,L) = a(C,L) = \frac{2}{L \cdot C}.
\end{equation*}
But since $C$ deforms to cover $X$, we also know that
\begin{equation*}
0 \leq (K_{X} + a(X,L)L) \cdot C.
\end{equation*}
Combining the two equations, we see that $-K_{X} \cdot C \leq 2$, and equality must be achieved since the family is dominant.  So the curves deform in dimension $\dim X - 1$, and the universal family map is indeed generically finite.  A similar calculation works for arbitrary families of subvarieties $Y$ breaking the balanced condition, assuming that there is a free rational curve on $Y$ with vanishing intersection against $a(X,L)L + K_{Y}$.
\end{exam}

\begin{prop} \label{genfinhigherdegree}
Let $X$ be a smooth projective uniruled variety of dimension $n$ and let $L$ be a big and nef divisor on $X$ such that $K_{X} + a(X,L)L$ is rigid.  Suppose that $p: \mathcal{U} \to W$ is a smooth family of varieties admitting a generically finite map $s: \mathcal{U} \to X$ which breaks the balanced property.  Then either the map $s: \mathcal{U} \to X$ is generically finite of degree $\geq 2$, or the monodromy action on $\mathcal GN^1(U/W)$ is non-trivial.
\end{prop}

\begin{proof}
Without loss of generality we may assume that $a(X, L)=1$.
If the monodromy action on $\mathcal GN^1(\mathcal{U}/W)$ is non-trivial then we are done.  Otherwise, let $\overline{\mathcal{U}}$ denote a smooth projective completion of $\mathcal{U}$ which admits a morphism $\overline{p}$ to a projective completion of $W$ that extends $p$ and a morphism $\overline{s}$ which extends $s$.  Run the relative $K_{\overline{\mathcal{U}}} + \overline{s}^{*}L$-MMP for the morphism $\overline{p}$.  Since $\overline{s}^{*}L$ is relatively big and nef, the result will be a birational contraction $\psi: \overline{\mathcal{U}} \dashrightarrow \overline{\mathcal{U}}'$ and a morphism $\overline{p}': \overline{\mathcal{U}}' \to \overline{W}$ whose general fiber is $K_{\overline{\mathcal{U}}'} + \psi_{*}\overline{s}^{*}L$-trivial.   Since each step is also a step of the $K_{\overline{\mathcal{U}}} + \overline{s}^{*}L$-MMP, we see that $\overline{\mathcal{U}}'$ has $\mathbb{Q}$-factorial terminal singularities.  Furthermore, by arguing as in the proof of Theorem \ref{bconstancy} we see the general fiber $F$ of $\overline{p}$ has $\mathbb{Q}$-factorial terminal singularities and the restriction map $N^{1}(\overline{\mathcal{U}}'/\overline{W}) \to N^{1}(F)$ is surjective (due to the monodromy assumption).   In particular $b(F,L) = \dim(N^{1}(F))$ for a general fiber $F$ of $p'$.  

Suppose for a contradiction that $\overline{s}$ is birational.  Let $\widetilde{X}$ be a smooth birational model of $X$ admitting birational maps $\rho: \widetilde{X} \to \overline{\mathcal{U}}'$ and $\psi: \widetilde{X} \to X$.  Consider the composition of maps
\begin{equation*}
f: N^{1}(\widetilde{X}) \to^{\rho_{*}} N^{1}(\overline{\mathcal{U}}') \to^{\textrm{restrict}} N^{1}(F).
\end{equation*}
Note that $f$ is surjective, so that the kernel $K$ of $f$ has codimension $b(F,L)$.  By Lemma \ref{rcfibers} the kernel of the second map has an effective basis; by the negativity of contraction lemma the kernel of the first map also has an effective basis.  Thus by pulling back, we see that the face $\mathcal{T}$ of $\Lambda_{\mathrm{eff}}(\widetilde{X})$ defined by taking the intersection with $K$ also has codimension $b(F,L)$.  Since $f$ preserves pseudo-effectiveness, $\mathcal{T}$ is a supported face of $\Lambda_{\mathrm{eff}}(\widetilde{X})$.

Note that $K_{\widetilde{X}} + \psi^{*}L$ is contained in the face $\mathcal{T}$.  Using the local polyhedrality of the pseudo-effective cone on a neighborhood of $K_{\widetilde{X}} + \psi^{*}L$, we see that this divisor lies in the interior of the minimal supported face containing it.  In particular, if this minimal face were $\mathcal{T}$, then $K_{\widetilde{X}} + \psi^{*}L$ would lie in the relative interior of $\mathcal{T}$, contradicting the rigidity of $K_{\widetilde{X}} + \psi^{*}L$.  Altogether, this shows that $b(X,L) > \mathrm{codim}(\mathcal{T}) = b(F,L)$.
\end{proof}

\section{Thin sets arising from subvarieties}
\label{sec: thinsets}
In this section we assume that our ground field $F$ is a number field.

\begin{prop}
\label{prop:birational}
Let $X$ be a geometrically uniruled smooth projective variety defined over a number field $F$ and $L$ a big and nef $\mathbb Q$-divisor on $X$ such that $a(X, L)L + K_X$ is rigid. Suppose that $\rho(X) = \rho(\overline{X})$.
Furthermore assume that we have a algebraic fiber space $f: X \rightarrow Y$ such that over an algebraic closure of $F$, $\overline{f} : \overline{X} \rightarrow \overline{Y}$ is a family of subvarieties breaking the balanced property for $L$. Let 
\[
Y^\circ = \{y \in Y \mid X_y \text{ is geometrically integral and smooth and $L|_{X_y}$ is big}\}
\]
Then the following set 
\[
\{y \in Y^\circ(F) \mid a(X, L) = a(X_y, L), \, b(X, L) \leq b(F, X_y, L)\}
\]
is a thin set in $Y^\circ$.
\end{prop}
\begin{proof}
Since $a(X, L)L+K_X$ is rigid, this adjoint divisor is numerically equivalent to a rigid effective divisor $D$. Let $V$ be the subspace of $N^{1}(X)$ spanned by the components of $D$. Then we have
$b(X, L) = \dim N^1(X)/V$.
For $y \in Y^\circ (F)$, consider the space $N^1(\overline{X}_y)$ and let $V'$ be the subspace spanned by geometric irreducible components of $D|_{X_y}$. Let $G$ be the finite group of the monodromy acting on $N^1(\overline{X}_y)$; this action is non-trivial by Proposition \ref{genfinhigherdegree}. Then we have the following diagram:
\[
N^1(X)/V \twoheadrightarrow N^1(\overline{X}_y)^G/V'^G \subset N^1(\overline{X}_y)/V'.
\]
Since $X$ and $X_y$ are both geometrically rationally connected, the second cohomology and the N\'eron-Severi space coincide, and the surjectivity of the first homomorphism follows from the corresponding statement for the second cohomology.  This homomorphism has a nontrivial kernel.
Let $\pi : (Y', y_1) \rightarrow (Y^\circ, y)$ be the \'etale covering defined over $F$ which kills the monodromy $G$. Then by the Hilbert Irreducibility Theorem, there exists a thin set $Z \subset Y^\circ(F)$ such that if $y' \in Y^\circ(F) \setminus Z$, the fiber $\pi^{-1}(y')$ is irreducible and the Galois action of the splitting field is $G$. For such $y'$, we have
\[
b(F, X_{y'}, L) = \dim N^1(\overline{X}_y)^G/V'^G   < \dim N^1(X)/V = b(X, L).
\]
Thus our assertion follows.
\end{proof}

\begin{exam}[Batyrev-Tschinkel's example] \label{btexample}
Consider the Fano variety over $\mathbb{Q}(\sqrt{-3})$:
\begin{equation*}
X := \left \{ \sum_{i=0}^{3} x_{i}^{3}y_{i} = 0 \right \} \subset \mathbb{P}^{3}_{\vec{x}} \times \mathbb{P}^{3}_{\vec{y}}.
\end{equation*}
\cite{BT} shows that the pair $(X,-K_{X})$ does not satisfy the closed set version of Manin's Conjecture due to the fibration by cubic surfaces.  Nevertheless, Proposition \ref{prop:birational} shows that the contribution of rational points from split fibers only forms a thin set of points on $X$.
(Here split fibers mean that fibers whose $b$-invariants are strictly greater than $1$.)
\end{exam}

\begin{coro}
\label{coro: thinset_eachfamily}
Let $X$ be a geometrically uniruled smooth projective variety defined over $F$ and $L$ a big and nef $\mathbb Q$-divisor on $X$ such that $a(X, L)L + K_X$ is rigid. Suppose that $\rho(X) = \rho(\overline{X})$.
Suppose that we have a dominant family of subvarieties $\pi : U \rightarrow W$ such that over $\overline{F}$, this is a family of subvarieties breaking the balanced property for $L$.  
We also assume that the evaluation map $s: U \rightarrow X$ is generically finite.
Define the set $Z$ by
\[
Z := \cup Y(F) \subset X(F)
\]
where $Y$ runs over members of $\pi$ defined over $F$ such that 
\begin{itemize}
\item $Y$ is not integral, or;
\item $L|_Y$ is not big, or;
\item $(a(X, L), b(X, L)) \leq (a(Y, L), b(F, Y, L))$.
\end{itemize}
Then $Z$ is a thin set.
\end{coro}

\begin{proof}
By applying Proposition \ref{genfinhigherdegree} to the base change to an algebraic closure, we see that either the evaluation map
has degree $\geq 2$ or the monodromy action is non-trivial (since both properties can be verified after this base change).
If the evaluation map has degree $\geq 2$, then our assertion is clear.
So we may assume that the evaluation map is birational.  The $Y$ satisfying the first two conditions form a closed subset. 
Take a resolution of singularities $\beta: \tilde{U} \rightarrow U$ and choose an open subset $W^\circ \subset W$ so that the morphism $\pi : (\pi\circ \beta)^{-1}(W^\circ) \rightarrow W^\circ$ is smooth.
Now our assertion follows from Proposition~\ref{prop:birational} by applying it to this smooth family.
\end{proof}


\begin{proof}[Proof of Theorem \ref{theo: mainI}:]
Define
\[
Z:= \cup Y(F) \subset X(F),
\]
where $Y$ runs over all geometrically integral subvarieties $Y$ defined over $F$ such that 
\begin{itemize}
\item $L|_Y$ is not big, or;
\item $(a(X, L), b(X, L)) \leq (a(Y, L), b(F, Y, L))$ in the lexicographic order.
\end{itemize}
We must show that $Z$ is contained in thin set.

Let $V$ be the closed set as in the statement of Theorem~\ref{theo: universal families},
and let $W$ be the closure in $\mathrm{Chow}(X)$ of the locus parametrizing subvarieties defined over $\overline{F}$ which are adjoint rigid, break the balanced property, and are not contained in $V$.
As shown in Theorem~\ref{theo: universal families}, $W$ is a bounded family.
Let $\pi : \mathcal U \rightarrow W$ be the universal family over $W$.
Let $W = \cup W_i$ be the decomposition of $W$ into irreducible components.

Suppose that $W_i$ is not geometrically irreducible.
Let $\pi_i : \mathcal U_i \rightarrow W_i$ be the universal family over $W_i$.  Then the $F$-rational points on $U_{i}$ must lie above a proper closed subset of $W_{i}$.  Thus $V_i = \overline{s_i(\mathcal U_i(F))}$ is a proper closed subset of $X$,
where $s_i: \mathcal U_i \rightarrow X$ is the evaluation map.
Note that the evaluation map must be generically finite if it is dominant because of Proposition~\ref{prop: genfinofmaps}.

Suppose that $W_i$ is geometrically irreducible.  Let $\pi_i : \mathcal U_i \rightarrow W_i$ be the universal family over $W_i$. When the evaluation map $s_i : \mathcal U_i \rightarrow X$ is not dominant, we define $V_i = s_i(\mathcal U_i)$, a proper closed subset of $X$.  Otherwise, the evaluation map $s_i : \mathcal U_i \rightarrow X$ is dominant, and $\overline{\pi_i} : \overline{\mathcal U}_i \rightarrow \overline{\pi_i}$ is a family of subvarieties breaking the balanced property.  Proposition \ref{prop: genfinofmaps} shows that $s_{i}$ is generically finite; it either has degree $\geq 2$ or it is birational.  In the first case, we take a smooth resolution $\tilde{\mathcal{U}}_i$ and we let $E_i$ be the ramification locus of $\tilde{s}_i : \tilde{\mathcal{U}}_i \rightarrow X$, and we define $Z_i = s_i(\mathcal U_i(F)) \cup \tilde{s}_i(E_i)$. This is obviously a thin set.  In the second case, let $K_i$ be a proper closed subset of $W_i$ that for any $w \in W_i\setminus K_i$, the member $Y_w$ of $\pi_i$ is geometrically integral and $a(Y_w, L)$ and $b(\overline{Y}_w, L)$ are constant.
Also we denote the exceptional locus of $s_i : \mathcal U_i \rightarrow X$ by $E_i$.

Let $Z_i= \bigcup Y(F) \cup s_i(\pi_i^{-1}(K_i)) \cup s_i(E_i)$ where $Y$ runs over all geometrically integral members of $\pi$, defined over $F$, such that
\begin{itemize}
\item $L|_Y$ is not big, or;
\item $(a(X, L), b(X, L)) \leq (a(Y, L), b(F, Y, L))$.
\end{itemize}
By Corollary~\ref{coro: thinset_eachfamily}, this set is thin.
We define $Z'$ by
\[
Z' = \bigcup V_i(F) \cup \bigcup Z_i\cup V(F)
\]
By construction this is a thin set.

We claim that $Z \subset Z'$.
Suppose that $Y$ is a geometrically integral subvariety such that
\begin{itemize}
\item $L|_Y$ is not big, or;
\item $(a(X, L), b(X, L)) \leq (a(Y, L), b(F, Y, L))$.
\end{itemize}
We need to verify that $Y(F)$ is contained in $Z'$. 
If $L|_Y$ is not big, then $Y$ is contained in $V$.
If $a(Y, L) > a(X, L)$, then again $Y$ is contained in $V$.
Hence we may assume that $a(X, L) = a(Y, L)$.
By Theorem \ref{theo: universal families} there exists $i$ and a geometrically integral subvariety $S \subset W_i$ such that
$s_i(\pi_i^{-1}(S)) = Y$ and $s_{i}$ is birational on $\pi_{i}^{-1}(S)$.
We may assume that $W_i$ is geometrically integral and $s_i :\mathcal U_i \rightarrow X$ is dominant, since in the other cases it is clear from the construction that $Y(F) \subset Z'$.  If $S \subset K_i$, then we are fine.  Otherwise, note that every point in $Y(F)$ is either contained in $s_{i}(E_{i})$, or is the $s_{i}$-image of a point in $\mathcal{U}_{i}(F)$, and suffices to consider points of the second type.  When $s_{i}$ has degree $\geq 2$, it is clear that such points are contained in $Z'$.  When $s_{i}$ is birational, then every fiber of $\pi_{i}^{-1}(S)$ not lying over $K_{i}$ satisfies the hypotheses of Corollary \ref{coro: thinset_eachfamily}, so again all such points are contained in $Z'$ by construction.
\end{proof}

\section{Surfaces}
\label{surfacesec}

In this section we assume that our ground field is an algebraically closed field of characteristic $0$.
We summarize the balanced properties for surfaces.

\begin{theo} \label{theo: balancedforsurfaces}
Let $S$ be a smooth uniruled projective surface and let $L$ be a big and nef $\mathbb{Q}$-divisor on $S$.  Suppose that:
\begin{itemize}
\item $\kappa(K_{S} + a(S,L)L) = 1$.  Then $L$ is weakly balanced with respect to general subvarieties and with respect to generically finite covers.
\item $\kappa(K_{S}+a(S,L)L) = 0$.  Then $L$ is balanced with respect to general subvarieties and with respect to generically finite covers of degree $\geq 2$.
\end{itemize}
\end{theo}

In the Iitaka dimension $1$ case, the geometry is particularly simple.  The fibers of the Iitaka fibration break the balanced property, and any finite cover breaking the balanced property is birational to a base change of the Iitaka fibration.  Thus we can expect the growth of rational points to be entirely controlled by the $\mathbb{P}^{1}$ fibers.

In the Iitaka dimension $0$ case, the following theorem complements the balanced property.

\begin{theo}
 \label{theo: finitenessofcovers}
 Let $S$ be a smooth uniruled projective surface and let $L$ be a big and nef $\mathbb{Q}$-divisor on $S$ such that $(S,L)$ is adjoint rigid.  Then any generically finite cover $f: Y \to S$ satisfying
 \begin{itemize}
 \item $a(Y,f^{*}L) = a(S,L)$, and
 \item $(Y,f^{*}L)$ is adjoint rigid
 \end{itemize}
 is birationally equivalent to a morphism which is \'etale in codimension $1$, and there are only finitely many such covers up to birational equivalence.

Moreover, if $S$ is a smooth del Pezzo surface, then there is no such generically finite cover of degree $\geq 2$.
\end{theo}

\begin{proof}[Proof of Theorem \ref{theo: surfaces}:]
Theorem \ref{theo: surfaces} follows immediately from the rigid case of Theorem \ref{theo: balancedforsurfaces} and from the first part of Theorem \ref{theo: finitenessofcovers}.
\end{proof}

The goal of this section is to prove Theorems \ref{theo: balancedforsurfaces} and \ref{theo: finitenessofcovers}.  We start with the balanced property.  Since the balanced property for subvarieties was addressed by \cite[Proposition 5.9]{LTT14}, we will focus on generically finite covers.  First we handle two special cases.

\begin{prop}
 Let $S = \mathbb P^2$ and $L = -K_S$. Then $L$ is strongly $a$-balanced with respect to
 any generically finite cover of degree $\geq 2$.
\end{prop}
\begin{proof}
 Use Proposition~\ref{threefoldbigbound}(2). Note that for any curve $C$ on $S$, we have $L\cdot C >2$.
\end{proof}

\begin{prop}
Let $S$ be a smooth weak del Pezzo surface of degree $d \leq 8$ and $L = -K_S$.
Then $L$ is balanced with respect to any generically finite cover $f : Y \rightarrow S$ of degree $\geq 2$.
\end{prop}

\begin{proof}
Note that $a(S,L) = 1$ and $b(S,L) = \rho(S) \geq 2$.  Without loss of generality we may suppose that $Y$ is smooth.  We have $a(Y,f^{*}L) \leq a(S,L)$ by the ramification formula.

Suppose we have a generically finite cover $f: Y \rightarrow S$ such that the adjoint divisor $K_{Y} + f^*(-K_{S})$ has positive Iitaka dimension.  Then $b(Y,f^{*}L) = 1 < b(S,L)$ and the balanced property automatically follows.  So henceforth we only consider the case when $K_{Y} + f^*L$ is rigid.  Note that the adjoint divisor $K_{Y} + f^*L$ is exactly the ramification divisor $R$ on $Y$.
 
 We apply the $(f^*L+K_Y)$-MMP to $Y$ and obtain a minimal model $\phi: Y \rightarrow Y'$.
 Then $Y'$ is a smooth weak del Pezzo surface and it follows from \cite[Lemma 3.10]{LTT14} that 
 \[
 b(Y, f^*L) = \rho(Y').
 \]
 Let $d'$ denote the degree of the smooth weak del Pezzo surface $Y'$ and let $e \geq 2$ be the degree of the generically finite cover $f:Y \rightarrow S$. We have
  \[
 de = (f^*L)^2 \leq (\phi^*\phi_*f^*L)^2 = (-K_{Y'})^2 = d'.
 \]
so that $d' > d$.  Thus $b(Y,f^{*}L) = \rho(Y') < \rho(S) = b(S, L)$ and $L$ is balanced with respect to $f$.
\end{proof}

Combining the previous two results shows that:
\begin{theo}
\label{theo: weakdelPezzo}
 Let $S$ be a smooth weak del Pezzo surface and $L = -K_X$.
 Then $L$ is balanced with respect to any generically finite cover of degree $\geq 2$.
\end{theo}

This special case is enough to deduce a statement for arbitrary big and nef divisors on surfaces as follows.

\begin{coro}
\label{coro: surfaces_genericallyfinitecovers}
 Let $S$ be a smooth uniruled projective surface and $L$ a big and nef $\mathbb Q$-divisor on $S$.
 Then $L$ is weakly balanced with respect to any generically finite cover $f : Y \rightarrow S$ of degree $\geq 2$.
 Moreover, if $a(S, L)L + K_S$ is rigid, then $L$ is balanced with respect to $f$.
\end{coro}
\begin{proof}
 Without loss of generality, we may assume that $Y$ is smooth.
 It follows from the ramification formula that $a(Y, f^*L) \leq a(S, L)$.
 
 Suppose that $a(Y, f^*L) = a(S, L) =:a$.
 If $aL + K_S$ has Iitaka dimension $1$, then $af^*L + K_Y$ also has the same Iitaka dimension.
 Thus we have $a(Y, f^*L) = a(S, L)$ and $b(Y, f^*L) = 1 = b(S, L)$.
 The divisor $L$ is weakly balanced with respect to $f$.
 
 Suppose that $aL+K_S$ is rigid.  We split into two cases depending on the rigidity of $af^{*}L + K_{Y}$.  First suppose that $af^*L +K_Y$ is not rigid.  Then 
 $Y$ admits a fibration whose general fiber $F$ satisfies $a(F, f^*L) =a$.
 Let $F' = f(F)$. Since $L$ is balanced with respect to curves passing through a general point,
 we must have $a = a(F, f^*L) = a(F',L)$ and $b(S, L) > b(F', L) = 1 = b(Y, f^*L)$.
 Thus $L$ is balanced with respect to $f$.
 
 Next suppose that $af^*L +K_Y$ is also rigid.
 We apply the $aL+K_S$-MMP to $S$ and we obtain a birational morphism $\phi : S \rightarrow S'$ to a smooth projective surface
 such that $aL' + K_{S'}\equiv 0$ where $L' = \phi_*L$.
 In particular, $S'$ is a smooth weak del Pezzo surface.
 We denote the composition $\phi \circ f$ by $g$.
 It follows from \cite[Lemma 3.5]{LTT14} that $a(S, L) = a(S', L')$ and $b(S, L) = b(S', L')$.
 By construction, we have $aL+K_S \equiv E$ where $E$ is an $\phi$-exceptional effective divisor
 whose support is equal to the $\phi$-exceptional locus.
 By our assumption, $af^*L + K_Y = af^*L + f^*K_S + R \equiv f^*E + R$ is a rigid divisor.
 
 On the other hand, it follows from the negativity lemma that $\phi^*L' = \phi^*\phi_*L = L+ E'$
 where $E'$ is another $\phi$-exceptional effective divisor such that $\mathrm{Supp}(E') \subset \mathrm{Supp}(E)$.
 Thus we have 
 \[
  ag^*L' + K_Y = af^*L + af^*E' + f^*K_S + R \equiv af^*E + af^*E' + R.
 \]
 Since $\mathrm{Supp}(E') \subset \mathrm{Supp}(E)$, this divisor is also rigid which implies $a(Y, f^*L) = a(Y, g^*L')$.
 Moreover, using \cite[Lemma 3.10]{LTT14}, we also have $b(Y, f^*L) = b(Y, g^*L')$.
 Now our assertion follows from Theorem~\ref{theo: weakdelPezzo} since $L'$ is numerically proportional to $-K_{S'}$.
\end{proof}


The rest of this section is devoted to a proof of Theorem~\ref{theo: finitenessofcovers}.  As with the balanced property, it is easy to reduce the statement to the case when $S$ is a smooth weak del Pezzo surface and $L = -K_{S}$.  We separate this situation into cases based on degree.

\begin{lemm} \label{deg5plus}
Let $S$ be a smooth weak del Pezzo surface of degree $\geq 5$ and let $L = -K_{S}$.  
If $f: Y \to S$ is a generically finite cover of degree $\geq 2$ such that $(Y,f^{*}L)$ is rigid then $a(Y,\pi^{*}L) < a(X,L)$.
\end{lemm}

\begin{proof}
Note that $\pi^{*}L^{2} \geq 10$.  Thus we conclude by Proposition \ref{threefoldbigbound}.(2) that $a(Y,\pi^{*}L) \leq \sqrt{9/10}$.
\end{proof}

Thus if we have a smooth weak del Pezzo surface of degree $\geq 5$, then 
there is no generically finite cover $f:Y \rightarrow S$ of degree $\geq 2$ such that
the ramification divisor $R$ is rigid.
However, for smooth weak del Pezzo surfaces of degree $\leq 4$ whose anticanonical line bundle is not ample, it is possible that there is such a cover of degree $\geq 2$.

\begin{theo} \label{genfinitesurf}
Suppose we have a commutative square of surface morphisms
\begin{equation*}
\xymatrix{Y \ar[r]^{f}\ar[d]_{\phi}&  S \ar[d]^{\psi}\\
Y' \ar[r]_{f'} & S'}
\end{equation*}
satisfying the following conditions:
\begin{itemize}
\item $S$ is a smooth weak del Pezzo surface, $S'$ has canonical singularities, and $\psi$ is a birational crepant map.
\item $Y$ and $Y'$ are smooth and $\phi$ is birational.
\item $f$ is generically finite.
\item The $\psi$-exceptional locus is contained in the $f$-image of the $\phi$-exceptional locus.
\item The ramification divisor $R$ for the morphism $f$ is rigid and $\phi$ is a sequence of steps in the $(K_{Y} - f^{*}K_{S})$-MMP.
\end{itemize}
Then $f$ is birationally equivalent to a morphism of normal surfaces with canonical singularities which is \'etale in codimension $1$.
\end{theo}

\begin{proof}
Set $L = -K_{S}$, a big and nef divisor.  Let $R'$ denote $\phi_{*}R$.  Since $\phi$ is $R$-negative,  $R'$ is also rigid and $R = \phi^{*}R' + F$ for some effective divisor $F$ whose support is equal to the $\phi$-exceptional locus.  Note furthermore that the $(K_{Y} - f^{*}K_{S})$-MMP terminates with a birational morphism that contracts all of $R$.

Let $r$ denote the number of components of  $\Supp(R)$.  The proof is by induction on the difference $r - \rho(Y/Y')$.
We first prove the base case.  Suppose that $r - \rho(Y/Y') = 0$, so that $R' = 0$.  Then the ramification locus for the map $\psi \circ f: Y \to S'$ consists entirely of divisors which are exceptional for the map.  Pushing forward, we see that the same is true for $f'$.  We conclude that the finite part of the Stein factorization $\widetilde{f}: \widetilde{Y} \to S'$ is \'etale in codimension $1$.  Since $S'$ has canonical singularities, $\widetilde{Y}$ does as well.

We now prove the induction step.  It suffices to show that if $R' \neq 0$ then we can find a further contraction $Y' \to Y''$ and $S' \to S''$ meeting the induction hypotheses.  Since $R'$ is the support of the locus contracted by the $K_{Y'} + \phi_{*}f^{*}L$-MMP and the pair $(Y', \phi_{*}f^{*}L)$ is terminal, there must be a $(-1)$-curve $C$ contained in the support of $R'$ which also has negative intersection against $R'$.  Note that $R'$ is the pushforward of the Cartier divisor $R$, so $R' \cdot C$ is an integer.  Since
\begin{equation*}
(K_{Y'} - f'^{*}K_{S'}) \cdot C \leq -1 \qquad K_{Y'} \cdot C = -1 \qquad -f'^{*}K_{S'} \cdot C \geq 0
\end{equation*}
we see that we must have an equality $-f'^{*}K_{S'} \cdot C = 0$.

Let $T$ denote the strict transform of $C$ on $Y$ and let $T'$ denote the support of the $\phi$-exceptional locus.  We have shown that any divisor supported on $f(T) \cup f(T')$ has vanishing intersection against $K_{S}$.  Thus, we can find a contraction $S \to S''$ contracting all such curves, where $S''$ has only canonical singularities.  By assumption, this map must contract every $\psi$-exceptional curve, so that this map factors through $S'$ (and possibly equals $\psi$). Since the map $Y' \to S' \to S''$ contracts $C$, it factors into a map $Y'' \to S''$.  These new maps satisfy all the hypotheses and we have decreased $r - \rho(Y/Y'')$.
\end{proof}

\begin{coro}
Let $S$ be a smooth weak del Pezzo surface.  
There is a fixed birational model $S'$ of $S$ with canonical singularities such that every generically finite $f: Y \to S$ with rigid ramification divisor $R$ is birational to a morphism $f': Y' \to S'$ which is \'etale in codimension $1$ and where $Y'$ has canonical singularities.

Thus there are only finitely many such $f$ up to birational equivalence.
\end{coro}

\begin{proof}
Let $S \to S'$ denote the contraction of all the $(-2)$-curves on $S$.

Let $f: Y \to S$ be any generically finite morphism and let $R$ denote the ramification divisor.  Applying the argument of Theorem \ref{genfinitesurf} to the case when $\phi$ and $\psi$ are the identity maps, we see that $f_{*}R$ is a union of $(-2)$-curves.  By applying the $K_{Y} + f^{*}(-K_{S})$-MMP to $Y$, we obtain a morphism $Y \to Y'$ contracting $R$ with $Y'$ smooth.  The map $Y \to S'$ factors through $Y'$. Then we take the Stein factorization $Y'' \rightarrow S'$ of $Y'\rightarrow S'$. This gives us an \'etale cover in codimension one.  Note that $Y''$ has only canonical singularities.

It then follows from \cite{GZ94} and \cite{KM99} that there are only finitely many possibilities for $f$ up to birational equivalence. 
\end{proof}

Finally, we pass from the special case of a smooth weak del Pezzo surface with the anticanonical divisor to the general case:

\begin{theo} \label{secondsurfacetheo}
 Let $S$ be a smooth uniruled projective surface and let $L$ be a big and nef $\mathbb{Q}$-divisor on $S$ such that $(S,L)$ is adjoint rigid.  Then any generically finite cover $f: Y \to S$ satisfying
 \begin{itemize}
 \item $a(Y,f^{*}L) = a(S,L)$, and
 \item $(Y,f^{*}L)$ is adjoint rigid
 \end{itemize}
 is birationally equivalent to a morphism which is \'etale in codimension $1$.  There are only finitely many such covers up to birational equivalence.
\end{theo}

\begin{proof}
By running the $a(S,L)L + K_{S}$-MMP, we find a birational morphism $\phi: S \to S'$ to a smooth projective surface such that $K_{S'} + a(S,L)\phi_{*}L \equiv 0$ and $a(S,L) = a(S',\phi_{*}L)$.  Since the pushforward of a big and nef divisor under a birational map of surfaces is still big and nef, we deduce that $S'$ is a smooth weak Del Pezzo.

Note that $\phi^{*}(-K_{S'}) - a(S, L)L$ is an effective $\phi$-exceptional divisor on $S$.  In particular, this difference is supported on $\Supp(E)$, where $E$ is the unique effective divisor numerically equivalent to $K_{S} + a(S,L)L$.  Then $f^{*}(\phi^{*}(-K_{S'}) - L)$ is contained in the support of the unique effective divisor numerically equivalent to $K_{Y} + a(S,L)f^{*}L = f^{*}E + R$.  Altogether we see that $a(Y,f^{*}L) = a(Y,f^{*}\phi^{*}(-K_{S'}))$ and $(Y, f^{*}\phi^{*}(-K_{S'}))$ is adjoint rigid.  Thus we have reduced to the weak del Pezzo case considered above.
\end{proof}

\begin{proof}[Proof of Theorem \ref{theo: finitenessofcovers}]
It only remains to show that if $S$ is a smooth del Pezzo surface and $L$ is a big and nef $\mathbb{Q}$-divisor on $S$, then there is no generically finite cover $f: Y \to S$ of degree $\geq 2$ such that $a(Y,f^{*}L) = a(S,L)$ and $(Y,f^{*}L)$ is adjoint rigid.  As in Theorem \ref{secondsurfacetheo}, we can contract a sequence of $(-1)$-curves to replace $(S,L)$ by $(S',-K_{S'})$ where $S'$ is another smooth del Pezzo.  The argument of Theorem \ref{genfinitesurf} shows that $R$ pushes forward to a union of $-2$-curves on a smooth del Pezzo.  There are no such curves, so the pushforward of $R$ to $S'$ is trivial.  We conclude that $f$ is birational to a finite morphism over $S'$ which is \'etale in codimension $1$, an impossibility due to the triviality of the fundamental group of $S'$.
\end{proof}

\section{Fano threefolds}
\label{fanothreefoldsec}

In this section we assume that our ground field is an algebraically closed field of characteristic $0$ unless otherwise specified.

\begin{defi}
Suppose $f: Y \dashrightarrow X$ is a rational map of normal $\mathbb{Q}$-factorial varieties with $\dim Y = \dim X$.  Let $p: W \to Y$ and $q: W \to X$ be a resolution of the map.  We say that $f$ is a generically finite contraction if every $p$-exceptional divisor is $q$-exceptional.  We say that $f$ is a birational contraction if furthermore $f$ is birational.
\end{defi}

Note that the definition is independent of the choice of resolution.  For a generically finite contraction, we define the pullback of a $\mathbb{Q}$-Cartier divisor $D$ on $X$ to be $f^{*}D := p_{*}q^{*}D$; clearly this is independent of the choice of resolution. 

\begin{prop}
\label{prop: length}
Let $X$ be a projective $\mathbb Q$-factorial terminal $3$-fold.
Suppose that we have an extremal $K_X$-negative divisorial contraction $f: X \rightarrow Y$
and we denote its exceptional divisor by $E$.
Then there is a family of rational curves $\Gamma \subset E$ covering $E$  and contracted by $f$ such that
\[
0< -K_X.\Gamma \leq 2.
\]
\end{prop}
\begin{proof}
The following argument was suggested to us by Kawakita.
When the divisorial contraction $f: X \rightarrow Y$ is a contraction to a curve,
our assertion follows from \cite{Kaw91}.
So we may assume that $f$ is a germ of a divisorial contraction to a point.
Such divisorial contractions are classified by Hayakawa, Kawakita, and Kawamata.
We go through the classification by Kawakita.


Let $S_Y$ be a general elephant of $Y$, i.e., a general element of the anticanonical linear system.
We denote the strict transform of $S_Y$ by $S_X$.
Then it follows from \cite[Theorem 1.5]{KawDuke05} that we have
\[
f^*S_Y = S_X + \frac{a}{n}E
\]
where $a/n$ is the discrepancy of $f$ and $n$ is the Cartier index of $p$.
In particular we have
\[
-K_X. (S_X \cap E) = \left(\frac{a}{n} \right)^2 E^3
\]
Note that since $S_Y$ has only canonical singularities, $S_X \cap E$ is a union of rational curves. On the other hand let $H_Y$ be a hyperplane section passing through $p$ and $H_X$ be its strict transform. Then we have
\[
f^*H_Y = H_X + bE
\]
where $b$ is a positive integer. Thus we have
\[
-K_X. (H_X \cap E) = \frac{a}{n}b E^3.
\]
When $P$ is Gorenstein, $H_Y$ has only canonical singularities,
so $(H_X \cap E)$ consists of rational curves.
In this case it suffices to show that
\[
\min \left\{ \left(\frac{a}{n} \right)^2 E^3, \quad  \frac{a}{n}b E^3\right\} \leq 2
\]
When $P$ is not Gorenstein $H_{Y}$ can have elliptic singularities, so we need to show that $( \frac{a}{n} )^2 E^3 \leq 2$.
We prove these inequalities by going through Kawakita's numerical classifications in \cite[Theorem 1.1]{KawDuke05}.

First we analyze exceptional type germs as in \cite[Theorem 1.1]{KawDuke05}.  By \cite[Theorem 1.3 and Table 3]{KawDuke05}, in all cases except possibly e1, e2, e3 we have that $(a/n)^2 E^{3} \leq 2$.
Suppose that our germ of divisorial contraction is of type e1.
Then $(a/n)E^3$ is equal to $4/r$ where $r \geq 4$ is an integer.
It follows from \cite[Theorem 1.3]{KawDuke05} that the discrepancy $a/n$ is either $1/n$, $1$, $2$, or $4$,
so $(a/n)^2E^3$ is less than equal to $2$ unless $a/n = 4$.
If $a/n = 4$, then we have $n=1$ so $P \in Y$ is Gorenstein.
It follows from \cite[Theorem 1.4]{KawJAMS03} that $b=1$, thus our assertion follows.
The cases e2 and e3 are handled in a similar way.

Suppose that our germ is of ordinary type.  The case o1 follows from \cite[Table 1, Corollary 2.4, Theorem 3.7]{KawDuke05} unless $n=1$.  But if $n=1$ then  $P$ is Gorenstein, so we use \cite[Theorem 1.4]{KawJAMS03} to conclude that $b=1$.
In the remaining cases o2 and o3 $b=1$ by the discussion following \cite[Equation (4.2)]{KawDuke05}.
Moreover it follows from \cite[Corollary 2.6]{KawDuke05} that the support of $H_X \cap E$ is a tree of $\mathbb P^1$s.  (In the application of the corollary, we use that for a germ $P \in Y$, $H_X$ is linearly equivalent to $nK_X - (a+1)E$). Thus our assertion follows.
\end{proof}

Let $X$ be a smooth Fano threefold of Picard rank $1$ and index $2$.
It is shown in \cite[Section 6]{LTT14} that if the degree $d = (-K_X)^3$ is $\geq 16$,
then the family $\pi : \mathcal U \rightarrow W$ of $-K_X$-conics forms the universal family of subvarieties breaking the balanced property for $L = -K_X$.
Using Proposition~\ref{prop: length}, one can show that all generically finite covers breaking the balanced property factor through the family $\mathcal U$.


\begin{proof}[Proof of Theorem \ref{theo: Fano3folds}: degree 2 case]
Suppose we have a commutative square of dominant rational maps
\begin{equation*}
\xymatrix{
Y \ar@{>}[r]^{f}\ar@{.>}[d]_{\phi}&  X \ar@{=}[d]\\
Y' \ar@{.>}[r]_{f'} & X}
\end{equation*}
satisfying the following conditions:
\begin{itemize}
\item $Y$ is smooth and $f$ is a generically finite morphism
such that $a(Y, f^*L) = 1$ and $f^*L +K_Y$ is rigid.
\item $f' : Y' \dashrightarrow X$ is a generically finite contraction.
\item $Y'$ has terminal singularities and $\phi$ is a birational contraction obtained as a sequence of steps in the $(K_{Y} - f^{*}K_{X})$-MMP.
\end{itemize}
We claim that the ramification divisor $R = K_{Y} - f^{*}K_{X}$ is contracted by a suitable choice of $\phi$
so that after taking the Stein factorization of a small $\mathbb{Q}$-factorial modification, the finite part $f' : \tilde{Y} \rightarrow X$
is \'etale in codimension one.
In particular since $X$ is smooth and simply connected, there is no such cover.

Indeed, let $r$ be the number of components in $R$.
We prove our assertion on the induction on the difference $r - \rho(Y/Y')$.
If $r-\rho(Y/Y') = 0$, then there is nothing to prove.
Suppose that $r-\rho(Y/Y') > 0$.
We further continue the $K_Y-f^*K_X$-MMP on $Y'$ until we obtain a divisorial contraction.  Since the hypotheses of the theorem are unchanged if we replace $Y'$ by a flip, we may for simplicity assume that the next step of the MMP-process is just a divisorial contraction $g: Y' \rightarrow \tilde{Y}$.
Let $D$ be the exceptional divisor contracted by $g$.
Then by Proposition~\ref{prop: length}, there is a family of rational curves $\Gamma$ covering $D$ and contracted by $g$ such that $-K_{Y'}. \Gamma \leq 2$.
Since $-\phi_*f^*K_X$ is movable and $(K_{Y'}-\phi_*f^*K_X). \Gamma < 0$
we have that $0 \leq -\phi_*f^*K_X .\Gamma < 2$.
Furthermore $-f^* K_X +E = -\phi^*\phi_*f^*K_X$ is true for some effective $\phi$-exceptional divisor $E$, so if $\tilde{\Gamma}$ denotes the strict transform of $\Gamma$ in $Y$ we have $-f^*K_X. \tilde{\Gamma} < 2$.

Since $X$ has index $\geq 2$, this is only possible when $-f^* K_X . \tilde{\Gamma} = 0$.  Thus $D$ is contracted by $f'$, and the map $\tilde{f}  = f' \circ g^{-1}$ is a generically finite contraction.
Now our assertion follows from the induction hypothesis.
\end{proof}

When $X$ is a smooth Fano threefold of Picard rank $1$, index $1$, and degree $\geq 10$, it is shown in \cite[Subsection 6.4]{LTT14} that the universal family of subvarieties breaking the balanced property for $-K_X$ is the family of $-K_X$-conics. Using Proposition~\ref{prop: length} one can prove a  statement about generically finite covers when $X$ is general in moduli:


\begin{proof}[Proof of Theorem \ref{theo: Fano3folds}: degree 1 case]
Since $X$ is general, the Fano variety of lines on $X$ is a curve of high genus  (see \cite{iskov}).
Let $T$ denote the divisor on $X$ swept out by $-K_{X}$-lines.
We claim that the evaluation map from the family of lines to this divisor is birational.  
Indeed, let $T'$ be a smooth resolution of $T$.  The strict transforms $F$ of $-K_{X}$-lines form a one dimensional family on $T'$, and in particular we have $-K_{T'}.F = 2$. This implies that $F^2=0$, and thus the strict transform of the lines defines a fibration on $D'$. This means that there is a morphism from $T'$ to the family $U$ of lines on $X$ realizing the birational equivalence.  In particular, since the Fano variety of lines on $X$ is a curve of high genus  (see \cite{iskov}) $T$ is not rationally connected.

We now follow the proof of the degree $2$ case.
Suppose we have a commutative square of dominant rational maps
\begin{equation*}
\xymatrix{
Y \ar@{>}[r]^{f}\ar@{.>}[d]_{\phi}&  X \ar@{=}[d]\\
Y' \ar@{.>}[r]_{f'} & X}
\end{equation*}
satisfying the conditions in 
the earlier proof. We claim that the ramification divisor $R = K_{Y} - f^{*}K_{X}$ is contracted by a suitable choice of $\phi$ so that after taking the Stein factorization, the finite part $f' : \tilde{Y} \rightarrow X$ is \'etale in codimension one. In particular since $X$ is smooth and simply connected, there is no such cover.

Indeed, let $r$ be the number of components in $R$.
Again we prove our assertion on the induction on the difference $r - \rho(Y/Y')$.
If $r-\rho(Y/Y') = 0$, then there is nothing to prove.
Suppose that $r-\rho(Y/Y') > 0$.
As 
before we continue the MMP, and without loss of generality the next step is a divisorial contraction $g: Y' \rightarrow \tilde{Y}$ of a divisor $D$.  Still arguing as before, we again deduce that $-f^*K_X. \tilde{\Gamma} < 2$ where $\tilde{\Gamma}$ is the strict transform of the family of rational curves constructed by Proposition \ref{prop: length}.  
Let $D$ be the exceptional divisor contracted by $g$.
Then by Proposition~\ref{prop: length}, there is a family of rational curves $\Gamma$ covering $D$ and contracted by $g$ such that $-K_{Y'}. \Gamma \leq 2$ because $-\phi_*f^*K_X$ is movable.
Let $\tilde{\Gamma}$ be the strict transform of $\Gamma$ in $Y$.
Since $-\phi_*f^*K_X$ is movable and $(K_{Y'}-\phi_*f^*K_X). \Gamma < 0$
we have that $0 \leq -\phi_*f^*K_X .\Gamma < 2$.
On the other hand $-f^* K_X +E = -\phi^*\phi_*f^*K_X$ is true for some $\phi$-exceptional divisor $E$, 
so we have $-f^*K_X. \tilde{\Gamma} < 2$.

Suppose first that $D$ is contracted to a point by $g$.  Then $D$ is rationally connected by \cite{HM07}.  However, the only divisor $T$ on $X$ swept out by $-K_{X}$-lines is not rationally connected as explained above.  
Thus in this case $D$ must be contracted by $f'$.  Suppose instead that $D$ is contracted to a curve by $g$.  Then $\Gamma'$ is the ruling on $D$ and $-K_{Y'} \cdot \Gamma' = 1$.  Arguing as above we see that $f^{*}K_{X} \cdot \tilde{\Gamma}' = 0$, and this can only be true if $D$ is contracted by $f'$.
In either case, the map $\tilde{f}  = f' \circ g^{-1}$ is a generically finite contraction. Now our assertion follows from the induction hypothesis.
\end{proof}

\begin{coro} \label{coro: fano3fold}
Let $X$ be a Fano $3$-fold over a number field $F$.  Assume that either $X$ has geometric index $\geq 2$, or that it has geometric Picard rank $1$ and index $1$ and its geometric model is general in the moduli.  Set $L = -K_{X}$.

As we vary $f$ over all thin $F$-maps $f: Y \to X$ such that either $f^{*}L$ is not big or
\begin{equation*}
(a(Y,f^{*}L),b(F, Y,f^{*}L)) \geq (a(X,L),b(F, X,L))
\end{equation*}
in the lexicographic order, the points
\begin{equation*}
\bigcup_{f} f(Y(F))
\end{equation*}
are contained in a thin subset of $X(F)$.
\end{coro}

\begin{proof}
Applying Theorem~\ref{theo: mainI} and Corollary~\ref{coro: surfaces_genericallyfinitecovers}, it only remains to consider the case when $f$ is surjective.  Then it is impossible for $\kappa(K_{Y} + f^{*}L) = 0$ by Theorem~\ref{theo: Fano3folds}.  Thus $\kappa(K_{Y}+f^{*}L) = 1$ or $2$.  After a birational modification $\beta: Y' \rightarrow Y$ we may assume that $Y'$ admits a fibration $\pi: Y' \to Z$ whose general fiber breaks the balanced condition on $Y$ over an algebraic closure.  If the fibers are surfaces, Corollary~\ref{coro: surfaces_genericallyfinitecovers} guarantees that the images of the fibers of $\pi$ form a family of subvarieties breaking the balanced condition on $X$.  (The analogous statement for curves is clear.)  Thus the Stein factorization of any such generically finite map must factor through the Stein factorization of a family of subvarieties breaking the balanced condition, and the points $Y(F)$ have already been accounted for.
\end{proof}

\section{Twists}
\label{sec: twists}
In this section, we assume that our ground field is a number field $F$.

Let $X$ be a smooth projective variety defined over $F$ and let $L$ be a big and nef $\mathbb Q$-divisor $L$ on $X$.  Even if we prove a finiteness statement for generically finite covers breaking the balanced property over an algebraic closure $\overline{F}$, the corresponding statement over the ground field $F$ might not be true due to the existence of twists -- many non-isomorphic varieties over $F$ can base-change to be isomorphic over $\overline{F}$.  To deduce a thinness statement for generically finite covers over $F$, we must consider the behavior of rational points for all twists of a fixed map.

Suppose now we fix a generically finite cover $f: Y \rightarrow X$ over $F$ where $Y$ is a smooth geometrically integral variety.  Conjecture \ref{mainconj} predicts that the rational points on the twists of $f$ which violate the balanced property are contained in a thin set.  Our goal is to prove this statement if one considers twists whose $a,b$-values are strictly greater than those of $X$.  If one allows equality of $a,b$-constants then Example \ref{exam: Sano} shows that the statement is false.   The example does not contradict the thin set version of Manin's conjecture: in fact, it is necessary to allow contributions from these twists to obtain Peyre's constant.  

\subsection{Twists} Let $f : Y \rightarrow X$ be a generically finite cover defined over $F$.
A twist of $f: Y \rightarrow X$ is a generically finite cover $f': Y' \rightarrow X$ such that after base change to $\overline{F}$ we have an isomorphism $g: Y_{\overline{F}} \cong Y'_{\overline{F}}$ with $\overline{f} = \overline{f'} \circ g$.

Recall that twists are parametrized by the Galois cohomology of $\mathrm{Aut}(Y/X)$:

\begin{prop}{\cite[Chapter III, Proposition 5]{Serre}}
Let $f: Y \rightarrow X$ be a generically finite cover defined over $F$ and let $\mathrm{Aut}(Y/X)$ denote the group of automorphisms of $Y$ which are compatible with $f$.
We assume that $Y$ is quasiprojective.
Then there is a one to one correspondence:
\[
\{\textnormal{isomorphism classes of twists of $f$}\} \cong H^1(\mathrm{Gal}(F), \mathrm{Aut}(Y/X))
\]
\end{prop}
For the reader's convenience we state this one to one correspondence explicitly.
Suppose that we have a twist $f' : Y' \rightarrow X$ of $f : Y \rightarrow X$.
We assume that this splits over an extension $F'/F$.
Thus we have an $X$-isomorphism $\phi : Y_{F'} \cong Y'_{F'}$.
Then an associated $1$-cocycle is given by
\[
\mathrm{Gal}(F'/F) \ni s \mapsto \phi^{-1}\circ \phi^s \in \mathrm{Aut}(Y_{F'}/X),
\]
where $\phi^s$ is the conjugate of $\phi$ by $s$.

Conversely if we have a $1$-cocycle $\mathrm{Gal}(F'/F) \ni s \mapsto \sigma_s \in \mathrm{Aut}(Y_{F'}/X)$, then we consider the action of $s \in \mathrm{Gal}(F'/F)$ on $Y\otimes {F'}$ by the composition of $\sigma_s$ and $1\otimes s$, then take the quotient 
\[
Y^\sigma =(Y\otimes F') / \mathrm{Gal}(F'/F).
\]
This is the twist of $Y$ by $\sigma$.

\subsection{Thinness for twists}

First we consider the case when a generically finite cover $f: Y \rightarrow X$ is not a Galois cover:
\begin{prop} \label{thinnessfornonGalois}
Let $f : Y \rightarrow X$ be a generically finite cover between geometrically integral projective varieties.
Suppose that over an algebraic closure, the extension $\overline{F}(Y)/\overline{F}(X)$ of function fields is not Galois. Then the following set
\[
\bigcup_{\sigma \in H^1(F, \mathrm{Aut}(Y/X))} f^\sigma(Y^\sigma (F)) \subset X(F)
\]
is contained in a thin subset of $X(F)$.
\end{prop}
\begin{proof}
We take a cover $g : \overline{W} \rightarrow \overline{Y}$ over an algebraic closure such that $\overline{F}(\overline{W})/\overline{F}(\overline{X})$ is the Galois closure of $\overline{F}(\overline{Y})/\overline{F}(\overline{X})$.
After shrinking $X$ if necessary, we may assume that $g\circ f$ is \'etale over $\overline{X}$.
Without loss of generality we may assume that $Y(F)$ is not empty after taking a finite base change.
Using the fundamental group exact sequence relating $\pi_{1}(Y)$ and $\pi_{1}(\overline{Y})$,
one can descend to a model $W$ defined over $F$ admitting a morphism $g : W \rightarrow Y$.
Then for any twist $Y^\sigma$ with $Y^\sigma(F) \neq \emptyset$, there is a twist $W^{\sigma'}$ of $W$ admitting a morphism $g': W^{\sigma'} \to Y^\sigma$.

Next after shrinking $X$ and taking a base change if necessary, we may assume that $G= \mathrm{Bir}(\overline{W}/\overline{X}) = \mathrm{Aut}(\overline{W}/\overline{X})$ and that the Galois action on this group is trivial (so that all the automorphisms are defined over the ground field).
It follows from the Hilbert Irreducibility Theorem that there exists a thin set $Z$ of $X(F)$ such that for any $P \not\in Z$, $(g\circ f)^{-1}(P)$ is irreducible and its Galois group is $G$. We claim that
\[
\bigcup_{\sigma \in H^1(F, \mathrm{Aut}(Y/X))} f^\sigma(Y^\sigma (F)) \subset Z.
\]
Indeed, if $P \in f^\sigma(Y^\sigma (F))$ for some $\sigma \in H^1(F, \mathrm{Aut}(Y/X))$,
then there exists $$\sigma' \in H^1(F, \mathrm{Aut}(W/X))$$ such that $g': W^{\sigma'}\rightarrow X$ factors through $Y^\sigma$ and $P \in g'(W^{\sigma'}(F))$ so
$(g\circ f)^{-1}(P)$ is a twist of $G$ by $\sigma'$. However since $g': W^{\sigma'}\rightarrow X$ factors through $Y^\sigma$, $\sigma'$ factors through a proper subgroup of G.
Thus $P \in Z$.
\end{proof}

The following theorem is useful to analyze Galois covers:

\begin{theo}{\cite[Theorem 1.7]{Cheltsov04}}
Let $f : Y \rightarrow X$ be a generically finite cover between projective varieties such that 
the function field extension $\overline{F}(Y)/\overline{F}(X)$ is Galois.
Assume that any birational transformation in $\mathrm{Bir}(\overline{Y}/\overline{X})$ is defined over the ground field $F$.
Then there is a birational model $Y'$ of $Y$ such that the finite group $\mathrm{Bir}(Y/X)$ acts regularly on $Y'$.
\end{theo}

\begin{prop} \label{thinnessforGalois}
Let $f : Y \rightarrow X$ be a generically finite cover between geometrically uniruled smooth projective varieties such that 
the function field extension $\overline{F}(Y)/\overline{F}(X)$ is Galois and $\rho(\overline{Y}) = \rho(Y)$.
We assume that $\mathrm{Bir}(\overline{Y}/\overline{X})= \mathrm{Aut}(Y/X)$, i.e., any birational transform of $Y$ over $X$ is regular and is defined over $F$.
Let $L$ be a big and nef $\mathbb Q$-divisor on $X$ such that $a(X, L)f^*L + K_Y$ is rigid.
Then the following set
\[
 Z= \bigcup_{\sigma} f^\sigma(Y^\sigma (F)) \subset X(F)
\]
is contained in a thin subset of $X(F)$.
Here $\sigma$ varies over all $\sigma \in H^1(F, \mathrm{Aut}(Y/X))$ such that
\[
(a(X, L), b(F, X, L)) < (a(Y, f^*L), b(F, Y^\sigma, (f^\sigma)^*L))
\]
holds in the lexicographic order.

Moreover, if the ramification divisor $R$ of $f$ contains an irreducible component $R'$ not contained in the support of the pullback of the rigid effective divisor $E$ numerically equivalent to $a(X, L)L + K_X$ and $R'$ is not exceptional, then the same thing holds assuming 
$\sigma$ varies over all $\sigma \in H^1(F, \mathrm{Aut}(Y/X))$ such that
\[
(a(X, L), b(F, X, L)) \leq (a(Y, f^*L), b(F, Y^\sigma, (f^\sigma)^*L))
\]
holds in the lexicographic order.
\end{prop}
\begin{proof}
We denote the exceptional divisors of $f$ by $E_1, \cdots, E_s$.
Since $X$ is the quotient of $Y$ by the group action of $G$, we have $N^1(\overline{Y})_{\mathbb Q}^G = N^1(\overline{X})_{\mathbb Q} \oplus (\bigoplus \mathbb Q E_i)^G$ and its Galois action is trivial because we assume that $\rho(\overline{Y}) = \rho(Y)$.
We choose a Zariski open set $X^\circ \subset X$ so that over $X^\circ$ the morphism $f$ is \'etale.
In particular, $f : Y^\circ = f^{-1}(X^\circ)\rightarrow X^\circ$ is a $G$-torsor.
Then we have
\[
X^\circ (F) = \bigsqcup_{\sigma \in H^1(F, G)} f^\sigma ((Y^\circ)^\sigma(F)),
\]
and $P \in  f^\sigma ((Y^\circ)^\sigma(F))$ if and only if the Galois action on $f^{-1}(P)$ is given via $\sigma$.
By the Hilbert Irreducibility Theorem, there is a thin set $Z \subset X(F)$ such that for any $P \not\in Z$, the fiber $f^{-1}(P)$ is irreducible and its Galois action is $G$.
Choose $\sigma \in \mathrm{Hom}(F, \mathrm{Aut}(Y/X))$ such that $\sigma : \mathrm{Gal}(F) \rightarrow \mathrm{Aut}(Y/X) = G$ is surjective.
Then since the Galois group acts on $N^1(\overline{Y})_{\mathbb Q}$ through $G$, it is easy to see that $b(F, Y^\sigma, (f^\sigma)^*L) \leq b(F, X, L)$.
Moreover if $R$ contains a non-exceptional component not contained in the support of $f^*E$, then the strict inequality holds.
Our assertion follows from these.
\end{proof}

\begin{proof}[Proof of Theorem \ref{thinnessfortwists}:]
Let $f: Y \to X$ be as in Theorem \ref{thinnessfortwists}.  We may prove thinness of point contributions of twists after a finite base change.  Furthermore, we may replace $Y$ by a higher birational model by including in the thin set all points in the image of the exceptional locus.  Then the result follows from Proposition \ref{thinnessfornonGalois} and Proposition \ref{thinnessforGalois}.
\end{proof}

In Proposition \ref{thinnessforGalois}, if all non-exceptional components of the ramification divisor of $f$ are contained in $f^*E$, then for all surjective homomorphisms $\sigma \in \mathrm{Hom}(\mathrm{Gal}(F), \mathrm{Aut}(Y/X))$ we have
\[
b(F, X, L)= b(F, Y^\sigma, (f^\sigma)^*L)
\]
and there are possibly infinitely many such twists achieving the equality (as in Example \ref{exam: Sano}).
We expect that we should include contributions from such twists in the counting function in order that Peyre's constant describes the growth rate of points.
More precisely:

\begin{ques}
Let $f : Y \rightarrow X$ be a generically finite cover between smooth projective varieties such that 
the function field extension $\overline{F}(Y)/\overline{F}(X)$ is Galois and $\rho(\overline{X}) = \rho(X)$.
We assume that $G = \mathrm{Bir}(\overline{Y}/\overline{X})= \mathrm{Aut}(Y/X)$, i.e., any birational transform of $Y$ over $X$ is regular and is defined over $F$.
Let $L$ be a big and nef $\mathbb Q$-divisor on $X$ such that $a(X, L)f^*L + K_Y \equiv E$ is rigid.
Suppose that any non-exceptional component of the ramification divisor is contained in $f^*E$.
Then do we have
\[
c(F, X, L) = \frac{1}{\# G}\sum_\sigma c(F, Y^\sigma, (f^\sigma)^*L)
\]
where $c(F, X, L)$ is Peyre's constant and $\sigma$ varies all elements in $\sigma \in H^1(F, G)$ such that
\[
(a(X, L), b(F, X, L)) = (a(Y, L), b(F, Y^\sigma, (f^\sigma)^*L))?
\]
\end{ques}

\subsection{Accumulation for twists}
In this section we show that the point contributions of twists with the same $a,b$-constants need not form a thin set.  The two examples below show that to obtain the correct Peyre's constant in Manin's Conjecture,  
sometimes we must discount such contributions and sometimes we must allow such contributions.

\begin{exam}[\cite{LeRudulier}] \label{lerudulierexample2}
Let $S = \mathbb P^2_{\mathbb{Q}}$ and $W$ be the blow up of $S \times S$ along the diagonal.
We denote the Hilbert scheme $\mathrm{Hilb}^{[2]}(S)$ of length $2$ subschemes by $X$,
and consider the following diagram:
\begin{equation*}
\xymatrix{W \ar[r]\ar[d]_{f}&  S \times S \ar[d]^{g}\\
X \ar[r]^{h} & \Chow^{2}(S)}
\end{equation*}
We let $L = -K_X$ which is a big and nef divisor on $X$.
Since the Picard rank of $X$ is $2$, we have
$
a(X, L) = 1, b(F, X, L) = 2.
$
On the other hand, since $h$ is a crepant resolution and $g$ is \'etale in codimension one,
we have $a(W, f^*L) = 1, b(F, W, f^*L) = 2$. Thus $L$ is not balanced with respect to the cover $f$.

It is clear that $\mathrm{Aut}(W/X)$ is isomorphic to $\mathbb Z/2$ and its Galois action is trivial.
So its first Galois cohomology is $\mathrm{Hom}(\mathrm{Gal}(F), \mathrm{Aut}(W/X))$ which is infinite because there are infinitely many quadratic extensions of $F$. Let \[\sigma \in \mathrm{Hom}(\mathrm{Gal}(F), \mathrm{Aut}(W/X))\] be a nontrivial element and consider the twist $W^\sigma$.  Letting $\overline{W}^{\sigma}$ denote the base-change to the algebraic closure, the Galois action on $\mathrm{NS}(\overline{W}^\sigma)$ is equal to the action of the involution of $W$ over $X$.  So the rank of $\mathrm{NS}(W^\sigma)$ is $2$.
On the other hand, $f : W^\sigma \rightarrow X$ is ramified along the divisor parametrizing non-reduced schemes, so ${f^\sigma}^*L + K_{W^\sigma}$ is equal to an irreducible exceptional divisor.
This implies that $b(F, W^\sigma, {f^\sigma}^*L) = 1$.

Thus the only twist violating the compatibility is the trivial one.
In \cite{LeRudulier}, Le Rudulier showed that after removing contributions from $W$ and the divisor parametrizing non-reduced subschemes, Manin's conjecture with Peyre's constant holds.
\end{exam}

\begin{exam}[\cite{Sano95}] \label{exam: Sano}
We learned the following example of a generically finite cover preserving $a$-values from \cite{Sano95}.
Let $Y = (\mathbb P^1)^{\times 3}$ and consider the involution $\iota$ given by
\[
\iota : [(x_0:x_1),(y_0:y_1),(z_0:z_1)] \mapsto [(-x_0:x_1),(-y_0:y_1),(-z_0:z_1)] 
\]
Let $X$ be the quotient of $Y$ by this involution $\iota$.
Then $X$ is a non-Gorenstein $\mathbb Q$-factorial terminal Fano 3-fold of Picard rank $3$.
The cover $f : Y \rightarrow X$ is \'etale in codimension one and its ramification locus consists of $8$ points on $Y$. Let $L = -K_X$. Then we have $a(X, L)= 1, b(F, X, L)= 3$ and $a(Y, f^*L) = 1, b(F, Y, f^*L) = 3$. Thus $L$ is not balanced with respect to the cover $f$. For any $\sigma \in H^1(F, \mathbb Z/2)$, its twist $Y^\sigma$ also satisfies $a(Y^{\sigma}, {f^\sigma}^*L) = 1, b(F, Y^\sigma, {f^\sigma}^*L) = 3$ and we have
\[
X(F) = \bigcup_{\sigma \in H^1(F, \mathbb Z/2)} f^\sigma (Y^\sigma(F)).
\]
Moreover $Y^\sigma$ is isomorphic to $Y$ as a $F$-variety,
so it contains a Zariski dense set of rational points.
Thus an analogous statement of Theorem~\ref{theo: mainI} for generically finite covers cannot hold if we include contributions from covers of the same $a, b$-values in the exceptional set.

However, this example does not contradict Manin's Conjecture.  Indeed, note that both $Y$ and $X$ are toric varieties,
so the closed-set version of Manin's conjecture with Peyre's constant is known for both varieties by \cite{BT-0} and \cite{BT-general}.

The only possible explanation is that the half of the summation of Tamagawa constants of twists of $Y$ is equal to the Tamagawa constant of $X$.  Thus it is necessary to include contributions of these twists to obtain Peyre's constant.

\end{exam}

\section{Examples}
\label{examplesec}

In this section, we assume that our ground field is an algebraically closed field of characteristic $0$.
\subsection{Projective space and quadrics}

We can understand the generically finite covers of projective space and quadric hypersurfaces using classical adjunction techniques.  Most references in adjunction theory work with ample instead of big and nef divisors, which is not sufficient for our purposes.  So we include the brief verifications here.

\begin{prop} \label{projspacebalanced}
Let $Y$ be a smooth projective variety of dimension $n$ and let $H$ be a big and nef Cartier divisor on $Y$.  
If $a(Y,H) = n+1$, then $Y \cdot H^{n} = 1$.  
\end{prop}


\begin{proof}
First note that if $a(Y,H)=n+1$ then $\kappa(K_{Y} + (n+1)H) = 0$.  Indeed, if the Iitaka dimension were at least one, then by \cite[Theorem 4.5]{LTT14} there would be a covering family of subvarieties breaking the balanced condition.  But for a smooth variety $F$ of dimension $<n$ and a big and nef divisor $H$ we always have that $K_{F} + (n+1)H$ is big by \cite[Proposition 2.10]{LTT14}.

Suppose that $Y$ satisfies $a(Y,H)=n+1$.  Note that $H^{i}(Y,K_{Y} + rH)$ vanishes for every $r>0$.  
Consider the polynomial defined on positive integers $r$ by $P(r) = \dim H^{0}(Y,K_{Y} + rH)$.  
We have that $P(i) = 0$ for $i=1,2,\ldots,n$ and $P(n+1)=1$.  Thus
\begin{equation*}
P(r) = \left( \begin{array}{c} r-1 \\ n \end{array} \right).
\end{equation*}
By Riemann-Roch we have $Y \cdot H^{n} = 1$.
\end{proof}

\begin{coro}
For any generically finite cover $\pi: Y \to \mathbb{P}^{n}$ from a smooth variety $Y$ of degree $\geq 2$ we have $a(Y,\pi^{*}H) < n+1$.
\end{coro}

\begin{proof}
If $a(Y,\pi^{*}H)=n+1$ then by Proposition \ref{projspacebalanced} we have $Y \cdot \pi^{*}H^{n} = 1$, a contradiction to the degree assumption.
\end{proof}

Combining with \cite[Example 5.1]{LTT14}, we have:

\begin{coro}
Projective space is strongly $a$-balanced with respect to any generically finite morphism and with respect to subvarieties.
\end{coro}

\begin{prop} \label{quadricbalanced}
Let $Q$ be a smooth quadric hypersurface $Q$ of dimension $n \geq 3$ and let $H$ denote the hyperplane class.  
Then $(Q,H)$ is strongly $a$-balanced with respect to any generically finite morphism and with respect to subvarieties.
\end{prop}

\begin{proof}
We first show that $Q$ is strongly $a$-balanced with respect to subvarieties.  
Arguing as in \cite[Example 5.3]{LTT14} it suffices to consider divisors.  
Let $D$ be the resolution of a divisor in $Q$ and suppose that $a(D,H)=n$.  Proposition \ref{projspacebalanced} shows that $H|_{D}$ has top self-intersection $1$.  So the image of $D$ in $Q$ must be a codimension $2$ hyperplane of $\mathbb{P}^{n+1}$, an impossibility once $\dim Q \geq 3$.

We next show that $Q$ is balanced with respect to covers.  
Let $\pi: Y \to X$ be any generically finite cover such that $a(Y,\pi^{*}H)=n$.  
If $\kappa(K_{Y} + n\pi^{*}H)=i \geq 1$, then (up to birational equivalence) $Y$ admits a fibration 
whose general fiber $F$ is a smooth $(n-i)$-fold with $a$-value $n$.  
The images of $F$ in $Q$ give a family of divisors breaking the $a$-balanced condition, an impossibility by the first paragraph.

Now suppose that $\kappa(K_{Y} + n\pi^{*}H) = 0$ and $a(Y,\pi^*H)=n$.  
Note that $H^{i}(K_{Y} + r\pi^{*}H)$ vanishes for every $r>0$.  
Consider the polynomial defined on positive integers $r$ by $P(r) = \dim H^{0}(Y,K_{Y} + r\pi^{*}H)$.  
We have that $P(i) = 0$ for $i=1,2,\ldots,n-1$ and $P(n)=1$. 

We claim that $P(n+1) = n+2$.  
Cutting down $Y$ by general hyperplane sections and using an easy inductive argument on LES of cohomology, 
it suffices to prove this when $Q$ and $Y$ have dimension $1$. 
Indeed, let $Y'$ be the pullback of a general hyperplane on $Q$.
By Bertini's theorem, we can assume that $Y'$ is smooth.
Using the exact sequence, we can prove that $\dim H^0(Y', K_{Y'} + r\pi^{*}H) = 0$ for any $0\leq r < n-1$, $\dim H^0(Y', K_{Y'} + (n-1)\pi^{*}H) = 1$
and $\dim H^0(Y', K_{Y'} + n\pi^{*}H) = P(n+1)-1$.
So we may assume that $Q$ is a smooth conic in $\mathbb P^2$ and $Y$ is a smooth curve
with the property $\dim H^0(Y, K_{Y}) = 0$, $\dim H^0(Y, K_{Y} + \pi^{*}H) = 1$ and $\dim H^0(Y, K_{Y} + 2\pi^{*}H) = P(n+1)-n+1$.
This implies that $Y$ is a smooth rational curve and $\mathcal O(H) = \mathcal O(2)$, so $P(n+1) = n+2$.
However we have $a(Q, H) = 1$ and $a(Y, H) =1$ so this contradicts with the fact that $\mathbb P^1$ is strongly $a$-balanced.
\end{proof}

\subsection{Toric varieties}

For toric varieties, the balanced property is established for toric subvarieties in \cite{HTT15}. Using Proposition~\ref{prop: genfinofmaps}, one can extend results in \cite{HTT15} to arbitrary subvarieties which are not contained in the boundary.

\begin{theo}
Let $X$ be a smooth projective toric variety and $L$ a big and nef $\mathbb Q$-divisor on $X$.
Let $Y$ be a subvariety of dimension $d$ on $X$ which is not contained in the boundary of $X$.
Then $L$ is weakly balanced with respect to $Y$.
Moreover if $a(X, L)L + K_X$ is rigid, then $L$ is balanced with respect to $Y$.
\end{theo}

Note that this geometric result is compatible with the results of \cite{BT-0} on Manin's conjecture for toric varieties, showing that one only needs to remove points in the torus-fixed locus.

\begin{proof}
Suppose that $L$ is not weakly balanced with respect to $Y$.
After replacing $Y$ by a general fiber of the Iitaka fibration from $Y$, we may assume that the pair $(Y, L)$ is adjoint rigid. We denote the torus acting on $X$ by $T$.
By translating $Y$ by $T$ we can form a family of subvarieties breaking the balanced property.
We denote the stabilizer of $Y$ by $T'$.
Then the induced map $T/T' \hookrightarrow \mathrm{Chow}(X)$ is injective,
so it follows from Proposition~\ref{prop: genfinofmaps} that the evaluation map must be generically finite. This means that $T'$ has dimension $d$ and $Y$ is a toric variety of $T'$.
The balanced property for toric subvarieties is established in \cite{HTT15}.
\end{proof}

\subsection{Le Rudulier's example} \label{lerudulierexample}


Let $S$ denote the surface $\mathbb{P}^{1} \times \mathbb{P}^{1}$, and let $H_{1}$ and $H_{2}$ denote the fibers of the two projections.  
Consider the Hilbert scheme $X := \Hilb^{2}(S)$; this is a weak Fano fourfold.  
We have $a(X,-K_{X}) = 1$ and $b(X,-K_{X})=3$.  Consider the variety $W$ which is the blow-up of $S \times S$ along the diagonal.  
These fit into a diagram
\begin{equation*}
\xymatrix{W \ar[r]\ar[d]_{f}&  S \times S \ar[d]^{g}\\
X \ar[r] & \Chow^{2}(S)}
\end{equation*}
where the vertical maps are the natural symmetric quotients.  Then $a(W,-f^{*}K_{X})=1$ and $b(W,-f^{*}K_{X}) = 4$, 
so that the predicted growth rate of points of bounded height on $W$ is larger than on $X$.  
In \cite{LeRudulier} Le Rudulier verifies that the $a,b$ constants for $W$ control the growth rate of rational points on $X$, 
but if you remove the thin set of points coming from $W$ then you recover the expected growth rate.

We will verify that $X$ is weakly balanced with respect to general subvarieties.  
Our main statement is

\begin{theo}
The anticanonical divisor $-K_X$ is weakly balanced with respect to general subvarieties,
but not balanced. Let $\pi_i : W \rightarrow S \times S \rightarrow \mathbb P^1$ be the projections to $\mathbb P^1$. Then the $\pi_i$, equipped with the natural map to $X$, are the universal families of subvarieties breaking the balanced property for $-K_X$.
\end{theo}
Thus Le Rudulier's example is fundamentally different from the example of \cite{BT-cubic} in  that the unexpectedly high rate of growth of rational points of bounded height can not be explained via subvarieties.  Henceforth we will freely use the results of \cite{BertramCoskun} describing the divisor theory of $X$.

\bigskip

\textbf{Divisors:} $N^{1}(X)$ is three dimensional; a basis is given by
\begin{itemize}
\item $H_{1}[2]$, the divisor parametrizing all subschemes where at least one point is contained if $F$ for some fixed fiber $F$ of the first projection.
\item $H_{2}[2]$, the analogous divisor for the second projection.
\item $E$, which is one-half of the divisor $B$ of non-reduced schemes.
\end{itemize}
The pseudo-effective cone is simplicial.  It is generated by $E$ and
\begin{itemize}
\item $D_{1} = H_{1}[2] - E$, the divisor parametrizing all subschemes supported on some (not fixed) fiber of the first projection.
\item $D_{2} = H_{2}[2] - E$, the analogous divisor for the second projection.
\end{itemize}
The nef cone is also simplicial and it coincides with the movable cone.  It is generated by $H_{1}[2]$, $H_{2}[2]$, and
\begin{itemize}
\item $X_{1,1} = H_{1}[2] + H_{2}[2] - E$, 
any degree $2$ subscheme determines a dimension $2$-subspace of $H^{0}(S,\mathcal{O}(1,1))$, 
and this divisor is the pullback of the corresponding ample generator of the Grassmannian $G(2,4)$.
\end{itemize}

\bigskip

\textbf{Curves:} $N_{1}(X)$ is three dimensional, and all the cones are simplicial.  The pseudo-effective cone is generated by:
\begin{itemize}
\item $J_{1}$ which is the set of non-reduced schemes (scheme-theoretically) contained in a fixed fiber of the first projection.
\item $J_{2}$ which is defined similarly.
\item $C$ which is the curve of double structures at a fixed point.
\end{itemize}
The nef cone is generated by:
\begin{itemize}
\item $F_{1}(1,2)$ which is the curve of subschemes containing one fixed point and letting the other point vary on some fixed fiber of the first projection.
\item $F_{2}(1,2)$, the analogue for the second projection.
\item $R(1,1)$ which is the curve of subschemes found by fixing a general element of $\mathcal{O}(1,1)$ (which is a rational curve on $S$), fixing a $2:1$ map to $\mathbb{P}^{1}$, and taking the fibers.
\end{itemize}

The intersection matrix for divisors and curves is:
\begin{equation*}
\begin{array}{c|c|c|c|c|c|c}
& X_{1,1} & H_{1}[2] & H_{2}[2] &  E & D_{1} & D_{2} \\ \hline
F_{1}(1,2) & 1 & 0 & 1 & 0 & 0 & 1 \\ \hline
F_{2}(1,2) & 1 & 1 & 0 & 0 & 1 & 0 \\ \hline
R(1,1) & 1 & 1 & 1 & 1 & 0 & 0 \\ \hline
C & 1 & 0 & 0 & -1 & 1 & 1 \\ \hline
J_{1} & 0 & 0 & 1 & 1 & -1 & 0 \\ \hline
J_{2} & 0 & 1 & 0 & 1 & 0 & -1 \\ \hline
P_{1} & 2 & 0 & 2 & 0 & 0 & 2 \\ \hline
P_{2} & 2 & 2 & 0 & 0 & 2 & 0
\end{array}
\end{equation*}

\bigskip

\textbf{Surfaces:}  $N_{2}(X)$ is 6-dimensional.  $\Sym^{2}(N^{1}(X))$ gives a five dimensional subspace (with relation $H_{1}[2]H_{2}[2] + E^{2} = H_{1}[2]E + H_{2}[2]E$).  Setting $FP$ to be the class of the surface parametrizing all degree $2$ subschemes of $S$ whose support contains a fixed point, we have:   
\begin{equation*}
\begin{array}{c|c|c|c|c|c|c}
& H_{1}[2]^{2} & H_{1}[2]H_{2}[2] & H_{2}[2]^{2} &  H_{1}[2]E & H_{2}[2]E & FP \\ \hline
H_{1}[2]^{2} & 0 & 0 & 2 & 0 & 0 & 0 \\ \hline
H_{1}[2]H_{2}[2] & 0 & 2 & 0 & 0 & 0 & 1 \\ \hline
H_{2}[2]^{2} & 2 & 0 & 0 & 0 & 0 & 0 \\ \hline
H_{1}[2]E & 0 & 0 & 0 & 0 & -2 & 0 \\ \hline
H_{2}[2]E & 0 & 0 & 0 & -2 & 0 & 0 \\ \hline
FP & 0 & 1 & 0 & 0 & 0  & 1
\end{array}
\end{equation*}


Before discussing the pseudo-effective and nef cones, we need to have a better understanding of the geometry of $X$. 
It carries several group actions.  First, it has an action of $\mathrm{PGL}_{2} \times \mathrm{PGL}_{2}$ 
arising from the fiberwise action on $\mathbb{P}^{1} \times \mathbb{P}^{1}$.  
This action has orbits $X-D_{1}-D_{2}-E$, $D_{1}-E|_{D_{1}}$, $D_{2}-E|_{D_{2}}$, $E - D_{1}|_{E} - D_{2}|_{E}$, $E \cap D_{1}$, $E \cap D_{2}$.  
Any irreducible surface which has the expected dimension of intersection with every orbit will be nef due to \cite[2 Theorem (i)]{Kle74}.

The variety $X$ also carries two natural involutions.  
There is an involution $\iota$ induced by the involution of $S$ which switches the two factors.  
This swaps $H_1[2]$ and $H_2[2]$ and fixes $E$ and the class of $FP$.  The other involution is induced by the birational map $X \to G(2,4)$.  
More precisely, let $V = H^0(S, \mathcal O(H_1+H_2))$ which is a $4$-dimensional vector space.
Fix a basis $s_1, s_2, s_3, s_4$ for $V$ and consider the symmetric pairing on $V$ 
whose intersection matrix with respect to $s_1, s_2, s_3, s_4$ is the identity.
This induces a natural involution on $\mathrm{Gr}(2, V)$ given by
\[
 \sigma : \mathrm{Gr}(2, V) \ni W \mapsto W^\perp \in \mathrm{Gr}(2, V).
\]
Let $m_1, m_2$ be a basis for $H^0(S, \mathcal O(H_1))$ and $l_1, l_2$ a basis for $H^0(S, \mathcal O(H_2))$.
We choose $s_1, s_2, s_3, s_4$ to be $m_1l_1, m_1l_2, m_2l_1, m_2l_2$.
The image $C_i$ of $D_i$ on $\mathrm{Gr}(2, V)$ is a smooth conic,
and the involution $\sigma$ fixes $C_i$.  Hence the involution $\sigma$ lifts to the involution on $X$.

We now identify several important classes of surfaces on $X$.  We will write these classes in terms of the basis $(H_{1}[2]^{2},H_{1}[2]H_{2}[2],H_{2}[2]^{2},H_{1}[2]E,H_{2}[2]E,FP)$ and describe the degree-$2$ subschemes parametrized by an open subset of a representative effective surface.  Recall that $H_{1}$, $H_{2}$ denote (any) fibers of the two projections on $S$.
\begin{equation*}
\begin{array}{c|c|c}
S_{1} & (1/2,0,0,-1/2,0,0) & \textrm{distinct points in a fixed $H_{1}$} \\ \hline
S_{2} & (0,0,1/2,0,-1/2,0) & \textrm{distinct points in a fixed $H_{2}$} \\ \hline
V_{1} & (0,0,0,1,0,0) & \textrm{non-reduced, contained set-theoretically in a fixed $H_{1}$} \\ \hline
V_{2} & (0,0,0,0,1,0) & \textrm{non-reduced, contained set-theoretically in a fixed $H_{2}$}\\ \hline
T_{1} & (0,1,0,0,-1,0) & \textrm{non-reduced, contained scheme-theoretically in any $H_{1}$} \\ \hline
T_{2} & (0,1,0,-1,0,0) & \textrm{non-reduced, contained scheme-theoretically in any $H_{2}$}  \\ \hline
FP & (0,0,0,0,0,1) &  \textrm{support contains a fixed point} \\ \hline
L_{1,2} & (0,1,0,0,0,-1) & \textrm{one point in each of a fixed $H_{1}$ and a fixed $H_{2}$} \\ \hline
L_{1,1} & (1,0,0,0,0,0) & \textrm{one point in each of two fixed $H_{1}$s}  \\ \hline
L_{2,2} & (0,0,1,0,0,0) & \textrm{one point in each of two fixed $H_{2}$s}  \\ \hline
M_{1,1} & (1/2,1,1/2,-1/2,-1/2,-1) & \textrm{the pullback of $\sigma_{1,1}$ from $G(2,4)$}  \\ \hline
M_{2} & (1/2,0,1/2,-1/2,-1/2,1) & \textrm{the pullback of $\sigma_{2}$ from $G(2,4)$} \\ \hline
\alpha_1 & (1/2,1,0,-1/2,0,-1) & \\ \hline
\alpha_2 & (0,1,1/2,0,-1/2,-1) & \\ \hline
\alpha_3 & (1/2,0,0,-1/2,0,1) & \\ \hline
\alpha_4 & (0,0,1/2,0,-1/2,1)
\end{array}
\end{equation*}


We claim that the cone of pseudo-effective surfaces $\mathrm{PEff}_2(X)$ is generated by 
\[
 S_1, S_2, V_1, V_2, T_1, T_2, FP, L_{1,2}.
\]
We verify the claim as follows.  Let $M$ be an irreducible surface in $X$.
If $M$ is contained in $D_1$, $D_2$, or $E$, it is easy to verify directly (using divisor theory on these projective bundles) that $M$ is contained in the cone generated by $S_1$, $S_{2}$, $T_{1}$, $T_{2}$, $V_{1}$, $V_{2}$.  Otherwise, $M$ meets with $X\setminus (D_1\cup D_2 \cup E)$.
In this case, for $i=1,2$ the intersection $M \cap D_i$ is $1$-dimensional, and we deduce that the intersection of $M$ with a general $S_i$ is $0$-dimensional because $S_i$ is base point free in $D_i$.  This implies that $M$ has nonnegative intersection with $S_i$.  Using a similar argument, $M$ has nonnegative intersection with $V_i$.  And since $FP$ and $L_{1,2}$ are nef classes, $M$ has nonnegative intersection with those classes.  Let $\mathcal{C}$ denote the cone generated by the $S_i$, $V_i$, $FP$, and $L_{1,2}$.  We have shown that in this case the class of $M$ is contained in $\mathcal{C}^{\vee}$, and one easily verifies that $\mathcal{C}^{\vee} \subset \mathcal{C}$.  This finishes the proof of the claim.  Dually, the nef cone is generated by
\[
 FP, L_{1,2}, L_{1,1}, L_{2,2}, M_{1,1}, M_{2,2}, \alpha_{1}, \alpha_{2}, \alpha_{3}, \alpha_{4}.
 \]




\bigskip

Recall that the movable cone of surfaces $\overline{\mathrm{Mov}}_2(X)$ is defined to be the closure of the cone generated by classes of effective surfaces such that every component deforms to cover $X$.  
By \cite[Remark 3.2]{FL13}, a necessary condition for $\alpha$ to be movable is that $\alpha \cdot L$ is a pseudo-effective curve class for every pseudo-effective divisor $L$.  Since the classes $S_i, V_i, T_i$ can be defined as the intersection of a pseudo-effective divisor class with a nef divisor class, 
we see that any movable class $\alpha$ must intersect with $S_i, V_i, T_i$ non-negatively.  This implies that $\alpha$ is nef, and we conclude that $\mathrm{Nef}_2(X)=\overline{\mathrm{Mov}}_2(X)$.


\subsubsection{Verifying balanced condition for subvarieties}

We next verify the weakly balanced condition with respect to subvarieties.  
Set $L = H_{1}[2] + H_{2}[2]$.  Recall that $K_{X} = -2H_{1}[2] - 2H_{2}[2]$, so $a(X,L) = 2$ and $b(X,L)=3$.

For simplicity we will only work with those subvarieties which deform to cover $X$ (although most likely it suffices to take $E$ as the exceptional set).  Using the group action, this is equivalent to focusing on subvarieties which meet $X  \backslash D_{1} \cup D_{2} \cup E$.  We will consider in turn subvarieties of dimension $=1,2,3$.

{\bf Curves}: For any rational curve $C$ on $X$ we have $a(C,L|_{C}) = \frac{2}{L \cdot C} \leq 2$ and that $b(C,L|_{C}) = 1$.  Thus $X$ is balanced with respect to curves.

It will be useful later on to classify the movable curves achieving $a=2$.  The movable classes have the form $\alpha = aF_1(1,2) + bF_2(1,2) + c R(1,1)$, and $\alpha$ is an integral class if and only if $a,b,c$ are integers.  In particular $a(C,L|_{C}) = 2$ exactly when the class of $C$ is $F_1(1,2)$ or $F_2(1,2)$.

{\bf Surfaces}: Let $Y$ be an irreducible surface which is not contained in $D_1\cup D_2 \cup E$. Then $Y$ is in the movable cone, so we have
\[
 Y \equiv aFP+bL_{1,1} + cL_{2,2} + dL_{1,2}+ eM_{1,1}+ fM_2+ g_1\alpha_1 + g_2\alpha_2+ g_3\alpha_3 + g_4\alpha_4,
\]
where all coefficients are nonnegative rational numbers. Moreover since $Y$ is an integral class, by taking intersection numbers with
$S_1, S_2, V_1, V_2, L_{1,2}, FP$, we conclude that $c, b, e+f+g_2+g_4, e+f+g_1+g_3, d+e+g_1+g_2, a+f+g_3+g_4$ are non-negative integers.
We have
\[
L^2.Y=2b+2c+(e+f+g_1+g_3)+(e+f+g_2+g_4)+2(a+f+g_3+g_4)+2(d+e+g_1+g_2).
\]
So only classes such that $L^2.Y <3$ are $L_{1,1}$, $L_{2,2}$, $FP$ or $L_{1,2}$.  Applying Theorem \ref{threefoldbigbound} to $(Y,L)$, we see that we can have $a(Y,L)=2$ only if either $(Y,a(Y,L)L)$ is not rigid or if $Y$ has class $L_{1,1}$, $L_{2,2}$, $FP$ or $L_{1,2}$.  In the first case, since $\kappa(Y,K_{Y} + a(Y,L)L) > 0$ we have $b(Y,L) = 1$ and the balanced condition is verified.

The surfaces of class $L_{1,1}$ or $L_{2,2}$ are fibers of the morphism $\mathrm{Sym}^2(S) \rightarrow \mathrm{Sym}^2(\mathbb P^1) \cong \mathbb P^2$ which is induced by
projections of $S$. Hence they are isomorphic to $\mathbb{P}^{1} \times \mathbb{P}^{1}$, so that $a(Y,L) = b(Y,L) = 2$.
When $Y$ is the class of $FP$, then the image of the morphism $X\rightarrow \mathrm{Sym}^2(S) \rightarrow \mathrm{Sym}^2(\mathbb P^1) \cong \mathbb P^2$ is a line $l$.  Thus, under the map $X \to \mathbb{P}^{2} \times \mathbb{P}^{2}$ induced by the projections we have that the image of $Y$ is $\mathbb{P}^{1} \times \mathbb{P}^{1}$.  An intersection calculation shows that this map is birational.  Since $L$ is pulled-back from the base, we can compute $a(Y,L) = b(Y,L) = 2$ directly.  Finally, the surfaces of class $L_{1,2}$ are mapped to the surfaces of class $FP$ under the involution $\sigma$, so they have the same $a,b$-values.

Altogether we see that $X$ is balanced with respect to surfaces.

{\bf Divisors}: \cite[Example 4.8]{BertramCoskun} verifies that the movable cone $\overline{\mathrm{Mov}}_3(X)$ of divisors coincides with the nef cone $\mathrm{Nef}_3(X)$.  So it is spanned by $H_{1}[2], H_{2}[2], H_{1}[2] + H_{2}[2] - E$.  We have $(2L)^{3} \cdot (aH_{1}[2] + bH_{2}[2] + c(H_{1}[2]+H_{2}[2] - E)) = 48a + 48b + 96c$.  By Corollary~\ref{threefoldbigboundimproved}, we only need to worry about divisors of class $H_{1}[2]$, $H_{2}[2]$, or divisors covered by curves of type $F_{1}(1,2)$ or $F_{2}(1,2)$.  We separate these into several cases:
\begin{itemize}
\item Divisors of class $H_{1}[2]$ or $H_{2}[2]$.  

Suppose that $Y \subset X$ is an irreducible divisor numerically equivalent to $H_1[2]$. 
Then $Y$ is a member of $H^0(X, \mathcal O(H_1[2]))$. Let $\pi : X \rightarrow \mathrm{Sym}^2(S)$ be the Chow morphism. 
Then we have $H^0(X, \mathcal O(H_1[2])) = H^0(\mathrm{Sym}^2(S), \mathcal O(H_1[2]))$, 
so $Y$ is the pullback of an irreducible divisor $Y'$ on $\mathrm{Sym}^2(S)$. 
Note that since the pullback is irreducible, we must have that $\Delta \not\subset Y'$ (where $\Delta$ denotes the image of the diagonal in $\Sym^{2}(S)$). 
We denote the first projection of $S$ by $p_1 : S\rightarrow \mathbb P^1$. This induces a morphism
\[
g : \mathrm{Sym}^2(S) \rightarrow \mathrm{Sym}^2(\mathbb P^1) \cong \mathbb P^2.
\]
This is the complete linear series of $H_1[2]$, so in particular, $Y'$ is the pullback of a line $l$ on $\mathbb P^2$. 
Also note that the image of $\Delta$ by $g$ is a smooth conic $C$ in $\mathbb P^2$.
For any $x \not\in C$, the fiber $g^{-1}(x)$ is isomorphic to $\mathbb P^1 \times \mathbb P^1$. 
For any $x \in C$, the fiber $g^{-1}(x)$ is a non-reduced scheme whose underlying reduced scheme is isomorphic to $\mathbb P^2$.
The group $\mathrm{PGL}_2$ acts on $\mathrm{Sym}^2(\mathbb P^1) \cong \mathbb P^2$,
and the action is transitive on $\mathbb P^2 \setminus C$ and $C$.
The group $\mathrm{PGL}_2$ also acts on $(\mathbb P^2)^\vee$, and it consists of two orbits:
the set of lines meeting $C$ transversally and the set of lines tangent to $C$.

Suppose that $Y'$ is the pullback of a line $l$ meeting $C$ transversally.
Then $Y'$ is singular along $Y'\cap \Delta$ which is scheme-theoretically the disjoint union of two smooth conics.
The divisor $Y \subset X$ is the blowup of $Y'$ along $Y'\cap \Delta$.
It follows from this description that $Y$ is smooth.
By the adjunction formula, we conclude that $a(Y, L|_Y) = 2$.
The semiample fibration associated to $aL|_Y +K_Y$ is a morphism $Y \rightarrow l$.
Let $x \in l$ be a general fiber. Then we have a homomorphism
\[
 h: \mathrm{NS}(Y) \rightarrow \mathrm{NS}(Y_x \cong \mathbb P^1 \times \mathbb P^1).
\]
If we denote the kernel of this map by $V$, then it follows from \cite[Proposition 2.18]{HTT15} that 
$b(Y, L|_Y) = \dim \, \mathrm{NS}(Y) - \dim \, V$. 
On the other hand, $\mathrm{NS}(Y_x)$ is isomorphic to $\mathbb Z^2$,
and there is a non-trivial monodromy $\mathbb Z/2$-action on the N\'eron-Severi space.
This implies that $b(Y, L|_Y) = 1$. Hence $L$ is balanced with respect to $Y$.

Suppose that $Y'$ is the pullback of a line tangent to $C$.
Then $Y'$ is singular along the non-reduced fiber, and in particular $Y'$ and $Y$ are not normal.
Let $\beta : \tilde{Y}' \rightarrow Y'$ be the normalization of $Y'$, then one can show that $\tilde{Y}'$ is smooth.
Because of the group action by $\mathrm{PGL}_2^2$, the divisor $Y'$ is movable. 
In particular it follows from \cite[Theorem 5.6]{LTT14} that $a(Y', L|_{Y'}) \leq 2$
so that $2\beta^*L + K_{\tilde{Y}'}$ is pseudo-effective.
On the other hand, let $F$ be a general fiber of the fibration $\tilde{Y}' \rightarrow l$.
Then we have
\[
 (2\beta^*L + K_{\tilde{Y}'})|_{F} = 2(H_2+H_4) + K_F =0. 
\]
Thus $2\beta^*L + K_{\tilde{Y}'}$ is not big, so we conclude that $a(Y', L|_{Y'}) =2$.
It is clear that $\tilde{Y}'$ has Picard rank $2$, so this implies that $b(Y',L|_{Y'}) \leq 2$.
Hence $L$ is balanced with respect to $Y$.

\item Irreducible divisors $Y$ covered by curves of the form $F_{1}(1,2)$ or $F_{2}(1,2)$ such that $f^{-1}(Y)$ is reducible.  
Note that $a(Y,L)=2$ in this case.  Choosing a component $Y'$ of the preimage, 
we see that the induced map $f: Y' \to Y$ is birational, and hence we can work with $Y'$ instead.  
Furthermore, since $L$ is pulled back from $(\mathbb{P}^{1})^{\times 4}$ and $Y'$ goes through a general point of $W$, 
we can instead work with the strict transform on $(\mathbb{P}^{1})^{\times 4}$.  
For simplicity we will call this strict transform $Y$, and we let $L$ the sum of the pullbacks of $\mathcal{O}(1)$ from the four components.

By symmetry, it suffices to suppose that $Y$ is covered by fibers of the first projection $(\mathbb{P}^{1})^{\times 4} \to (\mathbb{P}^{1})^{\times 3}$.  
So $Y = \mathbb{P}^{1} \times Z$ for some $Z \subset (\mathbb{P}^{1})^{\times 3}$.  
By passing to a birational model $Z'$, we may suppose that $Y' = \mathbb{P}^{1} \times Z'$ is smooth.  
Let $\psi: Y' \dashrightarrow Q$ be the map defined by the section ring of $K_{Y'} + 2L$.  We claim that $\psi$ factors birationally through $Z'$: 
letting $\rho_{1}$, $\rho_{2}$ denote the two projections on $Y$,
\begin{equation*}
K_{Y'} + 2L = \rho_{2}^{*}(K_{Z'} + 2H_{2} + 2H_{3} + 2H_{4}) + \rho_{1}^{*}(K_{\mathbb{P}^{1}} + 2H_{1}).
\end{equation*}
Taking $\rho_{2}$-sheaf pushforwards, we see the section ring of $K_{Y'} + 2L$ can be identified 
with the section ring of $K_{Z'} + 2H_{2} + 2H_{3} + 2H_{4}$.  
In particular, if we choose the surface $Z'$ so that it admits a morphism to the canonical model, then $Y'$ does as well, 
and one can show that
\begin{align*}
b(Y',L)-1 = &\textrm{ the codimension of the minimal face containing} 
\end{align*}
$$K_{Z'} + 2H_{2} + 2H_{3} + 2H_{4}.$$
When $K_{Z'} + 2H_{2} + 2H_{3} + 2H_{4}$ is big, then this codimension is zero, so our assertion follows.
Suppose that the Iitaka dimension of $K_{Z'} + 2H_{2} + 2H_{3} + 2H_{4}$ is less than $2$.
It follows from Proposition~\ref{threefoldbigbound} that this happens exactly when $Z$ is covered by fibers of the projection
from $(\mathbb P^1)^{\times 3}$ to $(\mathbb P^1)^{\times 2}$.
One can show that $b(Z', H_2+H_3+H_4) = 1$ unless $Y$ is equal to $(\mathbb P^1)^{\times 3}$.
When $Y$ is equal to $(\mathbb P^1)^{\times 3}$, $L$ is weakly balanced with respect to $Y$, but not balanced.

\item Irreducible divisors $Y$ covered by curves of the form $F_{1}(1,2)$ or $F_{2}(1,2)$ such that $f^{-1}(Y)$ is irreducible.  
Note that $a(Y,L)=2$.  Consider the projections $S \to \mathbb{P}^{1}$; 
they induce maps $\Hilb^{2}(S) \to \Hilb^{2}(\mathbb{P}^{1}) \cong \mathbb{P}^{2}$.  
Our condition on $Y$ is equivalent to saying that the image of $Y$ under one of these maps to $\Hilb^{2}(\mathbb{P}^{1})$ 
(without loss of generality the morphism corresponding to the first projection on $S$) is a curve $C$.  
We pass to a resolution $C'$ and consider the base changed $Y' \to C'$.  
Since $Y$ goes through a general point of $X$, the general fiber of this map is isomorphic to $\mathbb{P}^{1} \times \mathbb{P}^{1}$ 
and the map is generically smooth. 
We then compose with a resolution to obtain $\rho: \widehat{Y} \to C'$ which is isomorphic to $Y' \to C'$ over an open subset of $C'$.

Let $\psi: \widehat{Y} \dashrightarrow Q$ be the map induced by the section ring of $K_{\widehat{Y}} + 2L$.  
Similarly to the previous case, we argue that $\psi$ factors birationally through $\rho$.  
Note that $K_{\widehat{Y}} + 2H_{2}[2]$ is trivial along the general fiber.  
Applying the canonical bundle formula of \cite[Theorem 4.5]{fm00}, 
we see that there is an effective divisor $D$ on $C'$ such that $K_{\widehat{Y}} + 2H_{2}[2] - \rho^{*}(K_{C'} + D)$ is pseudo-effective.  
Thus we see that $K_{\widehat{Y}} + 2H_{1}[2] + 2H_{2}[2]$ is the pullback of an ample divisor on $C'$ unless the image $C$ is a line in $\mathbb{P}^{2}$.  
This special case is handled explicitly above.  
Otherwise, $\rho$ is the canonical morphism.  
Applying \cite[Theorem 4.5]{LTT14} we see that $b(Y',L)$ is bounded above by $b(F,L)$ for a general fiber $F$.  
Since $F$ is a quadric, we see that $b \leq 2$.  

\end{itemize}

The double cover $W$ admits a fibration to $\mathbb P^1$ and this provides us the universal families of subvarieties 
breaking the balanced property for $L$.

\nocite{*}
\bibliographystyle{alpha}
\bibliography{balancedIII}

\end{document}